\documentclass[12pt]{amsart}  

\usepackage[latin1]{inputenc}
\usepackage{amsmath} 
\usepackage{amsfonts}
\usepackage{amssymb}
\usepackage{stmaryrd}
\usepackage{latexsym} 
\usepackage{graphicx}
\usepackage{subfigure}
\usepackage{color}
\usepackage{hyperref}
\usepackage{verbatim}
\usepackage[all]{xy}
\usepackage{graphics}
\usepackage{pdfsync}
\usepackage{xcolor}
\usepackage{tikz}
\usepackage{bm}
\usetikzlibrary{arrows.meta}
\usetikzlibrary{shapes.geometric}

\theoremstyle{definition}
\newtheorem{theorem}{Theorem}[section]
\newtheorem{proposition}[theorem]{Proposition}

\newtheorem{lemma}[theorem]{Lemma}
\newtheorem{corollary}[theorem]{Corollary}
\newtheorem{problem}[theorem]{Problem}
\newtheorem{conjecture}[theorem]{Conjecture}
\newtheorem{definition}[theorem]{Definition}
\newtheorem{example}[theorem]{Example}

\newtheorem{observation}[theorem]{Observation}

\newtheorem{definition/theorem}[theorem]{Definition/Theorem}

\theoremstyle{remark}
\newtheorem{remark}[theorem]{Remark}

\numberwithin{equation}{section}
\newlength\cellsize 
\savebox2{%
\begin{picture}(15,15)
\put(0,0){\line(1,0){15}}
\put(0,0){\line(0,1){15}}
\put(15,0){\line(0,1){15}}
\put(0,15){\line(1,0){15}}
\end{picture}}
\newcommand\cellify[1]{\def\thearg{#1}\def\nothing{}%
\ifx\thearg\nothing
\vrule width0pt height\cellsize depth0pt\else
\hbox to 0pt{\usebox2\hss}\fi%
\vbox to 15\unitlength{
\vss
\hbox to 15\unitlength{\hss$#1$\hss}
\vss}}
\newcommand\tableau[1]{\vtop{\let\\=\cr
\setlength\baselineskip{-16000pt}
\setlength\lineskiplimit{16000pt}
\setlength\lineskip{0pt}
\halign{&\cellify{##}\cr#1\crcr}}}
\savebox3{%
\begin{picture}(15,15)
\put(0,0){\line(1,0){15}}
\put(0,0){\line(0,1){15}}
\put(15,0){\line(0,1){15}}
\put(0,15){\line(1,0){15}}
\end{picture}}
\newcommand\expath[1]{%
\hbox to 0pt{\usebox3\hss}%
\vbox to 15\unitlength{
\vss
\hbox to 15\unitlength{\hss$#1$\hss}
\vss}}
\newcommand\bas[1]{\omit \vbox to \cellsize{ \vss \hbox to \cellsize{\hss$#1$\hss} \vss}}

\usepackage[backend=bibtex]{biblatex}
\begin{document}

\title[A signed $e$-expansion of the chromatic quasisymmetric function]{A signed $e$-expansion of the chromatic quasisymmetric function}

\author{Foster Tom}
\thanks{Department of Mathematics, MIT}
\subjclass[2020]{Primary 05E05; Secondary 05E10, 05C15}
\keywords{chromatic quasisymmetric function, elementary symmetric function, natural unit interval graph, proper colouring, Shareshian--Wachs conjecture, Stanley--Stembridge conjecture}

\begin{abstract}
We prove a new signed elementary symmetric function expansion of the chromatic quasisymmetric function of any natural unit interval graph. We then use a sign-reversing involution to prove a new combinatorial formula for $K$-chains, which are graphs formed by joining cliques at single vertices. This formula immediately implies $e$-positivity and $e$-unimodality for $K$-chains. We also prove a version of our signed $e$-expansion for arbitrary graphs.
\end{abstract}


\maketitle
\vspace{65pt}
\section{Introduction}\label{section:introduction}

In 1995, Stanley \cite{chromsym} introduced the chromatic symmetric function $X_G(\bm x)$, a generalization of the chromatic polynomial of a graph $G$, and famously conjectured with Stembridge \cite{chromsym, stanstem} that $X_G(\bm x)$ is $e$-positive whenever $G$ is the incomparability graph of a $(\bm 3+\bm 1)$-free partially ordered set. Guay-Paquet \cite{stanstemreduction} showed that it suffices to prove $e$-positivity for the smaller class of natural unit interval graphs. Shareshian and Wachs \cite{chromposquasi} defined a further generalization called the chromatic quasisymmetric function $X_G(\bm x;q)$, showed that if $G$ is a natural unit interval graph, then $X_G(\bm x;q)$ is in fact symmetric; and conjectured that it is $e$-positive and $e$-unimodal for such graphs. This notorious problem has received considerable attention. Many authors proved $e$-positivity for certain families of graphs \cite{csfvertex, chrombounce3, chromstoe, chromqrook, trilad, lollilari, chromquasidi, eposclasses, lltunicellschur, eposclaw2k2free, stanstem211, cyclechordepos, somepos, chrominvkostka}, studied related positivity problems \cite{chromquasipower, chromnsymschurposets, chromesinks, chromincreasingforests, eposclawcon, chromposets, chromspostrees, chromposnsym, chrompos31, stanstemreduction, chromeschers, cylindric31}, and explored alternative generalizations of the chromatic symmetric function \cite{symincreasingforests, genchrom, tuttechromsymeq, kromatic, extendedchromsym, chromnsym, graphrelated}. The Stanley--Stembridge conjecture is also closely related to positivity of immanants of Jacobi--Trudi matrices \cite{stanstem}, the cohomology of regular semisimple Hessenberg varieties \cite{chromhesssplitting, dothessenberg, chromhessequi, hessenberghopf, stanstemhess, chromquasihessenberg}, and characters of Hecke algebras evaluated at Kazhdan--Lusztig basis elements \cite{chromhecke, chromcharacters}.\\

In this paper, we prove a new signed elementary symmetric function expansion of the chromatic quasisymmetric function of any natural unit interval graph, in terms of objects called \emph{forest triples}. We then use a sign-reversing involution on forest triples to prove a new combinatorial $e$-expansion for $K$-chains, which are graphs formed by joining a sequence of cliques at single vertices. \\

In Section \ref{section:background}, we introduce chromatic quasisymmetric functions and natural unit interval graphs. In Section \ref{section:foresttripleformula}, we prove our signed formula (Theorem \ref{thm:foresttriples})
\begin{equation}
X_G(\bm x;q)=\sum_{\mathcal F\in\text{FT}(G)}\text{sign}(\mathcal F)q^{\text{weight}(\mathcal F)}e_{\text{type}(\mathcal F)}.
\end{equation}
In Section \ref{section:kchains}, we use a sign-reversing involution on forest triples to prove the explicit combinatorial $e$-expansion (Corollary \ref{cor:kchainexplicit}) for a generalization of $K$-chains, namely
\begin{equation}\label{eq:charmbraceletformula}
X_{K_\gamma^\epsilon}(\bm x;q)=[\gamma_1-2]_q!\cdots[\gamma_\ell-2]_q!\sum_{\alpha\in A_\gamma^\epsilon}[\alpha_1]_q\prod_{i=2}^{\ell+1}q^{m_i}[|\alpha_i-(\gamma_{i-1}-\epsilon_{i-1}-1)|]_qe_{\text{sort}(\alpha)}.
\end{equation}
This formula implies that $X_{K_\gamma^\epsilon}(\bm x;q)$ is $e$-positive and $e$-unimodal (Corollary \ref{cor:kchainuni}). In Section \ref{section:further}, we propose some new directions of study for the Stanley--Stembridge conjecture and we prove a version of our forest triple formula for arbitrary graphs (Theorem \ref{thm:nbctriples}).

\section{Background}\label{section:background}

Let $G=([n],E)$ be a graph with vertex set $[n]=\{1,\ldots,n\}$. A \emph{colouring} of $G$ is a function $\kappa:[n]\to\mathbb N=\{1,2,3,\ldots\}$ and we say that $\kappa$ is \emph{proper} if $\kappa(i)\neq\kappa(j)$ whenever $\{i,j\}\in E$, in other words, adjacent vertices are assigned different positive integers, which we think of colours. We denote by $\text{asc}(\kappa)$ the number of \emph{ascents} of $\kappa$, which are pairs of vertices $(i,j)$ with $i<j$ and $\kappa(i)<\kappa(j)$. The \emph{chromatic quasisymmetric function} of $G$ is the formal power series in an infinite alphabet of variables $\bm x=(x_1,x_2,x_3,\ldots)$ and an additional variable $q$ given by \cite[Definition 1.2]{chromposquasi}

\begin{equation}
X_G(\bm x;q)=\sum_{\substack{\kappa:[n]\to\mathbb N\\\kappa \text{ proper}}}q^{\text{asc}(\kappa)}x_{\kappa(1)}x_{\kappa(2)}\cdots x_{\kappa(n)}.
\end{equation}

Unlike the \emph{chromatic symmetric function} $X_G(\bm x)=X_G(\bm x;1)$, which is always symmetric in the $x_i$ variables, $X_G(\bm x;q)$ is not symmetric in general. Nevertheless, if $G$ is a \emph{natural unit interval graph}, meaning that \begin{equation}\label{eq:uicondition}\text{ for all } 1\leq i<j<k\leq n\text{, if }\{i,k\}\in E,\text{ then }\{i,j\}\in E\text{ and }\{j,k\}\in E,\end{equation} then $X_G(\bm x;q)$ is symmetric in the $x_i$ variables \cite[Theorem 4.5]{chromposquasi} and palindromic in $q$ with center of symmetry $\frac{|E|}2$ \cite[Corollary 4.6]{chromposquasi}, meaning that $X_G(\bm x;q)=q^{|E|}X_G(\bm x;q^{-1})$. In this case, we can consider the expansion
\begin{equation}X_G(\bm x;q)=\sum_\mu c_\mu(q)e_\mu\end{equation} in terms of the \emph{elementary symmetric functions} $e_\mu$, which are indexed by integer partitions $\mu=\mu_1\cdots\mu_\ell$ and defined by \begin{equation}e_\mu=e_{\mu_1}\cdots e_{\mu_\ell},\text{ where }e_k=\sum_{i_1<\cdots<i_k}x_{i_1}\cdots x_{i_k}.\end{equation} We say that $X_G(\bm x;q)$ is \emph{$e$-positive} if every polynomial $c_\mu(q)$ has positive coefficients and we say that $X_G(\bm x;q)$ is \emph{$e$-unimodal} if every polynomial $c_\mu(q)$ is \emph{unimodal}, meaning that the coefficients form a weakly increasing sequence followed by a weakly decreasing sequence.

\begin{figure}\caption{\label{fig:chromsymexamplebowtie} The bowtie graph $G$ and the chromatic quasisymmetric function $X_G(\bm x;q)$}
\begin{tikzpicture}
\draw (1.5,0) node (1) {$G=$};
\filldraw (2.134,0.5) circle (3pt) node[align=center,above] (1){1};
\filldraw (2.134,-0.5) circle (3pt) node[align=center,below] (2){2};
\filldraw (3,0) circle (3pt) node[align=center,above] (3){3};
\filldraw (3.866,0.5) circle (3pt) node[align=center,above] (4){4};
\filldraw (3.866,-0.5) circle (3pt) node[align=center,below] (5){5};
\draw (2.134,0.5)--(2.134,-0.5)--(3,0)--(2.134,0.5) (3.866,0.5)--(3.866,-0.5)--(3,0)--(3.866,0.5);
\end{tikzpicture}
\begin{align}\label{eq:bowtiee}
X_G(\bm x;q)=&(q^2+2q^3+q^4)e_{32}+(q+3q^2+4q^3+3q^4+q^5)e_{41}\\&\nonumber+(1+3q+4q^2+4q^3+4q^4+3q^5+q^6)e_5
\end{align}
\end{figure}
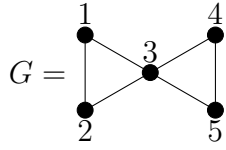

\begin{example}
The bowtie graph $G$ and the chromatic quasisymmetric function $X_G(\bm x;q)$ are shown in Figure \ref{fig:chromsymexamplebowtie}. We see that $G$ is a natural unit interval graph because it satisfies \eqref{eq:uicondition} and we see from \eqref{eq:bowtiee} that $X_G(\bm x;q)$ is indeed palindromic in $q$ with center of symmetry $\frac{|E(G)|}2=\frac 62=3$, and it is $e$-positive and $e$-unimodal. 
\end{example}

\begin{example}\label{ex:complete}
For the complete graph $K_n$ on $n$ vertices, a proper colouring $\kappa$ must use $n$ distinct colours and can be associated with a permutation $\sigma\in S_n$ by setting $\sigma(i)=j$ if $\kappa(i)$ is the $j$-th smallest colour used. Therefore, we have
\begin{equation}\label{eq:complete}
X_{K_n}(\bm x;q)=\left(\sum_{\sigma\in S_n}q^{|\{(i,j): \ i<j, \ \sigma_i<\sigma_j\}|}\right)e_n.
\end{equation}Then defining the \emph{$q$-integer} and the \emph{$q$-factorial}
\begin{equation}
[k]_q=1+q+q^2+\cdots+q^{k-1}=\frac{q^k-1}{q-1}\text{ and }[n]_q!=[n]_q[n-1]_q\cdots[2]_q[1]_q,
\end{equation}we have by a standard induction argument that the sum in \eqref{eq:complete} is $[n]_q!$, and so
\begin{equation}
X_{K_n}=[n]_q!e_n.
\end{equation}
Therefore $X_{K_n}(\bm x;q)$ is $e$-positive and $e$-unimodal.
\end{example}

Stanley was interested in characterizing $e$-positivity of $X_G(\bm x)$ in order to generalize positivity results of immanants of Jacobi--Trudi matrices \cite{stanstem}. The celebrated Stanley--Stembridge conjecture \cite[Conjecture 5.1]{chromsym}, \cite[Conjecture 5.5]{stanstem} asserts that we have $e$-positivity for incomparability graphs of $(\bm 3+\bm 1)$-free partially ordered sets and Guay-Paquet \cite[Theorem 5.1]{stanstemreduction} showed that it suffices to prove this for the smaller class of natural unit interval graphs. The following quasisymmetric refinement by Shareshian and Wachs has seen connections to Hessenberg varieties and Hecke algebras.

\begin{conjecture}\label{conj:ss}
\cite[Conjecture 1.3]{chromposquasi} If $G$ is a natural unit interval graph, then $X_G(\bm x;q)$ is $e$-positive and $e$-unimodal.
\end{conjecture}

\begin{remark}
An abstract graph is called a \emph{unit interval graph} if its vertices can be labelled $1$ through $n$ so that \eqref{eq:uicondition} holds. The name comes from an equivalent characterization as a graph whose vertices are unit intervals of the real line and whose edges join intersecting intervals. It is convenient to consider \emph{natural unit interval graphs}, where the adjective ``natural'' means that such a labelling is given. The chromatic quasisymmetric function does not depend on the choice of labelling, as long as \eqref{eq:uicondition} is satisfied \cite[Proposition 4.2]{chromposquasi}.
\end{remark}

\begin{remark} \label{rem:reverse} Let $G=([n],E)$ be a natural unit interval graph and consider the graph $G'=([n],\{\{n+1-j,n+1-i\}: \ \{i,j\}\in E\})$ given by reversing the vertex labels. If we restrict to $n$ variables and associate a proper colouring $\kappa:[n]\to[n]$ of $G$ to $\kappa':[n]\to[n]$ by setting $\kappa'(i)=n+1-\kappa(n+1-i)$, it follows that $X_G(x_1,\ldots,x_n;q)=X_{G'}(x_n,\ldots,x_1;q)$. Then we have $X_G(\bm x;q)=X_{G'}(\bm x;q)$ by symmetry in the $x_i$ variables. This proves palindromicity in $q$ and shows that we could have equivalently defined $X_G(\bm x;q)$ using descents instead of ascents. Similarly, we can switch the roles of $i$ and $(n+1-i)$ in everything that follows, for example, we can use increasing trees instead of decreasing trees.
\end{remark}

\section{Signed formula}\label{section:foresttripleformula}

In this section, we prove a signed combinatorial $e$-expansion of the chromatic quasisymmetric function of any natural unit interval graph $G=([n],E)$. This formula has the advantage of allowing us to calculate a particular term without calculating the entire chromatic quasisymmetric function. We will walk through an example in Figure \ref{fig:ftexample}.

\begin{definition}
A subtree $T$ of $G$ is \emph{decreasing} if for every vertex $v$ of $T$, the unique path in $T$ from the largest vertex of $T$ to $v$ is decreasing. Equivalently, every vertex of $T$ except the largest has exactly one larger neighbour. \\

Given a decreasing subtree $T$ of $G$, we define a permutation of $V(T)$, denoted $\text{list}(T)$, by reading the vertices of $T$ starting from the smallest, and at each step, reading the smallest unread vertex of $T$ that is adjacent to a read vertex of $T$. An example is given in Figure \ref{fig:treelists}.
\end{definition}

\begin{remark}
These trees and permutations also arise in work of Abreu and Nigro \cite{symincreasingforests}, Athanasiadis \cite{chromquasipower}, and D'Adderio, Riccardi, and Siconolfi \cite{chromincreasingforests} in their study of chromatic quasisymmetric functions.
\end{remark}

\begin{figure} \caption{\label{fig:treelists} A decreasing tree $T$ and the permutation $\text{list}(T)$}
$$\begin{tikzpicture}
\draw (-1,0) node () {$T=$};
\filldraw (0,0) circle (3pt) node[align=center, above] (1){1};
\filldraw (1,0) circle (3pt) node[align=center, above] (5){5};
\filldraw (1.5,-0.866) circle (3pt) node[align=center, left] (3){3};
\filldraw (2,0) circle (3pt) node[align=center,above] (6){6};
\filldraw (3,0) circle (3pt) node[align=center,above] (7){7};
\filldraw (4,0) circle (3pt) node[align=center,above] (8){8};
\filldraw (2.5,-0.866) circle (3pt) node[align=center, left] (4){4};
\filldraw (2,-1.73) circle (3pt) node[align=center,below] (2){2};
\draw (0,0) -- (4,0) (2,0) -- (1.5,-0.866) (2,0) -- (2.5,-0.866) (1.5,-0.866) -- (2,-1.73);
\draw (8,0) node () {$\text{list}(T)=15632478$};
\end{tikzpicture}$$
\end{figure}

\begin{definition}\label{def:ft}
A \emph{tree triple} of $G$ is an object $\mathcal T=(T,\alpha,r)$ consisting of the following data.
\begin{itemize}
\item $T$ is a decreasing subtree of $G$.
\item $\alpha=\alpha_1\cdots\alpha_\ell$ is an integer composition with size $|\alpha|=\sum_{i=1}^\ell\alpha_i=|V(T)|$.
\item $r$ is a positive integer with $1\leq r\leq \alpha_1$, the first part of $\alpha$. 
\end{itemize}
A \emph{forest triple} of $G$ is a sequence of tree triples $\mathcal F=(\mathcal T_i=(T_i,\alpha^{(i)},r_i))_{i=1}^m$ such that each vertex of $G$ is in exactly one tree $T_i$ and we have
\begin{equation}
\min(V(T_1))<\min(V(T_2))<\cdots<\min(V(T_m)).
\end{equation} The \emph{type} of a forest triple $\mathcal F$ is the integer partition \begin{equation}\text{type}(\mathcal F)=\text{sort}(\alpha^{(1)}\cdots\alpha^{(m)}),\end{equation} 
obtained by concatenating the compositions and sorting in decreasing order. The \emph{sign} of $\mathcal F$ is the integer \begin{equation}\text{sign}(\mathcal F)=(-1)^{\sum_{i=1}^m(\ell(\alpha^{(i)})-1)}=(-1)^{\ell(\text{type}(\mathcal F))-m}.\end{equation} Letting $\sigma=\text{list}(T_1)\cdots\text{list}(T_m)$ be the concatenation of the permutations of the $V(T_i)$, we define the \emph{weight} of $\mathcal F$ to be the integer
\begin{align}
\text{weight}(\mathcal F)&=\text{inv}_G(\sigma)+\sum_{i=1}^m(r_i-1),\text{ where }\\\nonumber\text{inv}_G(\sigma)&=|\{(i,j): \ i<j, \ \sigma_i>\sigma_j, \ \{\sigma_j,\sigma_i\}\in E(G)\}|\end{align}
is the number of \emph{$G$-inversions} of the permutation $\sigma\in S_n$. We denote by $\text{FT}(G)$ the set of forest triples of $G$ and by $\text{FT}_\mu(G)$ the set of forest triples of $G$ of type $\mu$.
\end{definition}

We now state our signed combinatorial formula.

\begin{theorem}\label{thm:foresttriples}
The chromatic quasisymmetric function $X_G(\bm x;q)$ of a natural unit interval graph $G$ satisfies
\begin{equation}\label{eq:ftformula}
X_G(\bm x;q)=\sum_{\mathcal F\in\text{FT}(G)}\text{sign}(\mathcal F)q^{\text{weight}(\mathcal F)}e_{\text{type}(\mathcal F)}=\sum_\mu\left(\sum_{\mathcal F\in\text{FT}_\mu(G)}\text{sign}(\mathcal F)q^{\text{weight}(\mathcal F)}\right)e_\mu.
\end{equation}
\end{theorem}

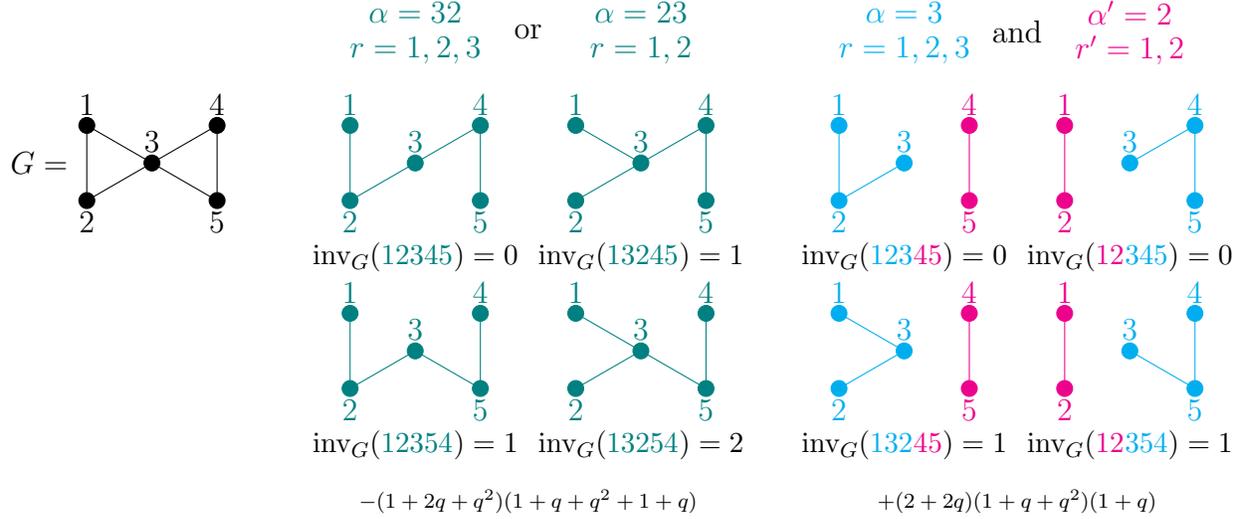
\begin{figure}\caption{\label{fig:ftexample} The bowtie graph $G$ and the forest triples of $G$ of type $32$}
\begin{tikzpicture}
\draw (3,0) node (1) {$G=$};
\filldraw (3.634,0.5) circle (3pt) node[align=center,above] (1){1};
\filldraw (3.634,-0.5) circle (3pt) node[align=center,below] (2){2};
\filldraw (4.5,0) circle (3pt) node[align=center,above] (3){3};
\filldraw (5.366,0.5) circle (3pt) node[align=center,above] (4){4};
\filldraw (5.366,-0.5) circle (3pt) node[align=center,below] (5){5};
\draw (3.634,0.5)--(3.634,-0.5)--(4.5,0)--(3.634,0.5) (5.366,0.5)--(5.366,-0.5)--(4.5,0)--(5.366,0.5);

\draw (8,2) node [color=teal] (){$\alpha=32$};
\draw (8,1.5) node [color=teal] (){$r=1,2,3$};
\draw (9.5,1.75) node [](){$\text{or}$};
\draw (11,2) node [color=teal] (){$\alpha=23$};
\draw (11,1.5) node [color=teal] (){$r=1,2$};
\draw (9.5,-4.5) node [](){\tiny{$-(1+2q+q^2)(1+q+q^2+1+q)$}};

\filldraw [color=teal](7.134,0.5) circle (3pt) node[align=center,above,color=teal] (1){1};
\filldraw [color=teal](7.134,-0.5) circle (3pt) node[align=center,below,color=teal] (2){2};
\filldraw [color=teal](8,0) circle (3pt) node[align=center,above,color=teal] (3){3};
\filldraw [color=teal](8.866,0.5) circle (3pt) node[align=center,above,color=teal] (4){4};
\filldraw [color=teal](8.866,-0.5) circle (3pt) node[align=center,below,color=teal] (5){5};
\draw [color=teal](7.134,0.5)--(7.134,-0.5)--(8,0) (8.866,0.5)--(8.866,-0.5) (8,0)--(8.866,0.5);
\draw (8,-1.25) node (){\small{$\text{inv}_G(\textcolor{teal}{12345})=0$}};

\filldraw [color=teal](7.134,-2) circle (3pt) node[align=center,above,color=teal] (1){1};
\filldraw [color=teal](7.134,-3) circle (3pt) node[align=center,below,color=teal] (2){2};
\filldraw [color=teal](8,-2.5) circle (3pt) node[align=center,above,color=teal] (3){3};
\filldraw [color=teal](8.866,-2) circle (3pt) node[align=center,above,color=teal] (4){4};
\filldraw [color=teal](8.866,-3) circle (3pt) node[align=center,below,color=teal] (5){5};
\draw [color=teal](7.134,-2)--(7.134,-3)--(8,-2.5) (8.866,-2)--(8.866,-3) (8,-2.5)--(8.866,-3);
\draw (8,-3.75) node (){\small{$\text{inv}_G(\textcolor{teal}{12354})=1$}};

\filldraw [color=teal](10.134,0.5) circle (3pt) node[align=center,above,color=teal] (1){1};
\filldraw [color=teal](10.134,-0.5) circle (3pt) node[align=center,below,color=teal] (2){2};
\filldraw [color=teal](11,0) circle (3pt) node[align=center,above,color=teal] (3){3};
\filldraw [color=teal](11.866,0.5) circle (3pt) node[align=center,above,color=teal] (4){4};
\filldraw [color=teal](11.866,-0.5) circle (3pt) node[align=center,below,color=teal] (5){5};
\draw [color=teal](10.134,0.5)--(11,0)--(10.134,-0.5) (11.866,0.5)--(11.866,-0.5) (11,0)--(11.866,0.5);
\draw (11,-1.25) node (){\small{$\text{inv}_G(\textcolor{teal}{13245})=1$}};

\filldraw [color=teal](10.134,-2) circle (3pt) node[align=center,above,color=teal] (1){1};
\filldraw [color=teal](10.134,-3) circle (3pt) node[align=center,below,color=teal] (2){2};
\filldraw [color=teal](11,-2.5) circle (3pt) node[align=center,above,color=teal] (3){3};
\filldraw [color=teal](11.866,-2) circle (3pt) node[align=center,above,color=teal] (4){4};
\filldraw [color=teal](11.866,-3) circle (3pt) node[align=center,below,color=teal] (5){5};
\draw [color=teal] (10.134,-2)--(11,-2.5)--(10.134,-3) (11.866,-2)--(11.866,-3) (11,-2.5)--(11.866,-3);
\draw (11,-3.75) node (){\small{$\text{inv}_G(\textcolor{teal}{13254})=2$}};

\draw (14.5,2) node [color=cyan] (){$\alpha=3$};
\draw (17.5,2) node [color=magenta] (){$\alpha'=2$};
\draw (16,1.75) node [](){$\text{and}$};
\draw (14.5,1.5) node [color=cyan] (){$r=1,2,3$};
\draw (17.5,1.5) node [color=magenta] (){$r'=1,2$};
\draw (16,-4.5) node [](){\tiny{$+(2+2q)(1+q+q^2)(1+q)$}};

\filldraw [color=cyan](13.634,0.5) circle (3pt) node[align=center,above,color=cyan] (1){1};
\filldraw [color=cyan](13.634,-0.5) circle (3pt) node[align=center,below,color=cyan] (2){2};
\filldraw [color=cyan](14.5,0) circle (3pt) node[align=center,above,color=cyan] (3){3};
\filldraw [color=magenta](15.366,0.5) circle (3pt) node[align=center,above,color=magenta] (4){4};
\filldraw [color=magenta](15.366,-0.5) circle (3pt) node[align=center,below,color=magenta] (5){5};
\draw [color=cyan](13.634,0.5)--(13.634,-0.5)--(14.5,0); \draw [color=magenta](15.366,0.5)--(15.366,-0.5);
\draw (14.5,-1.25) node (){\small{$\text{inv}_G(\textcolor{cyan}{123}\textcolor{magenta}{45})=0$}};

\filldraw [color=cyan](13.634,-2) circle (3pt) node[align=center,above,color=cyan] (1){1};
\filldraw [color=cyan](13.634,-3) circle (3pt) node[align=center,below,color=cyan] (2){2};
\filldraw [color=cyan](14.5,-2.5) circle (3pt) node[align=center,above,color=cyan] (3){3};
\filldraw [color=magenta](15.366,-2) circle (3pt) node[align=center,above,color=magenta] (4){4};
\filldraw [color=magenta](15.366,-3) circle (3pt) node[align=center,below,color=magenta] (5){5};
\draw [color=cyan](13.634,-2)--(14.5,-2.5)--(13.634,-3); \draw [color=magenta](15.366,-2)--(15.366,-3);
\draw (14.5,-3.75) node (){\small{$\text{inv}_G(\textcolor{cyan}{132}\textcolor{magenta}{45})=1$}};

\filldraw [color=magenta](16.634,0.5) circle (3pt) node[align=center,above,color=magenta] (1){1};
\filldraw [color=magenta](16.634,-0.5) circle (3pt) node[align=center,below,color=magenta] (2){2};
\filldraw [color=cyan](17.5,0) circle (3pt) node[align=center,above,color=cyan] (3){3};
\filldraw [color=cyan](18.366,0.5) circle (3pt) node[align=center,above,color=cyan] (4){4};
\filldraw [color=cyan](18.366,-0.5) circle (3pt) node[align=center,below,color=cyan] (5){5};
\draw [color=magenta](16.634,-0.5)--(16.634,0.5); 
\draw [color=cyan](17.5,0)--(18.366,0.5)--(18.366,-0.5);
\draw (17.5,-1.25) node (){\small{$\text{inv}_G(\textcolor{magenta}{12}\textcolor{cyan}{345})=0$}};

\filldraw [color=magenta](16.634,-2) circle (3pt) node[align=center,above,color=magenta] (1){1};
\filldraw [color=magenta](16.634,-3) circle (3pt) node[align=center,below,color=magenta] (2){2};
\filldraw [color=cyan](17.5,-2.5) circle (3pt) node[align=center,above,color=cyan] (3){3};
\filldraw [color=cyan](18.366,-2) circle (3pt) node[align=center,above,color=cyan] (4){4};
\filldraw [color=cyan](18.366,-3) circle (3pt) node[align=center,below,color=cyan] (5){5};
\draw [color=magenta](16.634,-3)--(16.634,-2); 
\draw [color=cyan](17.5,-2.5)--(18.366,-3)--(18.366,-2);
\draw (17.5,-3.75) node (){\small{$\text{inv}_G(\textcolor{magenta}{12}\textcolor{cyan}{354})=1$}};
\end{tikzpicture}
\end{figure}

\begin{example} Let us calculate the coefficient of $e_{32}$ in the chromatic quasisymmetric function $X_G(\bm x;q)$ for the bowtie graph $G$ in Figure \ref{fig:ftexample}. Forest triples $\mathcal F\in\text{FT}_{32}(G)$ come in one of two flavours.\\

We could have a single tree triple $\mathcal T=(T,\alpha,r)$, where $T$ is a decreasing spanning tree of $G$ and $\alpha$ is either $32$ or $23$. The four possibilities of $T$ are given in Figure \ref{fig:ftexample} along with $\text{inv}_G(\sigma)$. We can independently choose either $\alpha=32$ and $r=1$, $2$, or $3$, or $\alpha=23$ and $r=1$ or $2$, and we have $\text{sign}(\mathcal F)=(-1)^{2-1}=-1$, so these forest triples contribute 
\begin{equation}\label{eq:bowtieeneg}-(1+2q+q^2)(1+q+q^2+1+q)e_{32}=-(1+q)^2(2+2q+q^2)e_{32}.\end{equation} Alternatively, $\mathcal F$ can consist of two tree triples $\mathcal T=(T,\alpha,r)$ and $\mathcal T'=(T',\alpha',r')$ with trees of sizes $3$ and $2$ and compositions the single parts $3$ and $2$. The four possibilities of $(T,T')$ are given in Figure \ref{fig:ftexample} along with $\text{inv}_G(\sigma)$. We can independently choose $r$ and $r'$, and we have $\text{sign}(\mathcal F)=(-1)^{(1-1)+(1-1)}=+1$, so these forest triples contribute
\begin{equation}\label{eq:bowtieepos}
+(2+2q)(1+q+q^2)(1+q)e_{32}=+(1+q)^2(2+2q+2q^2)e_{32}.
\end{equation} Therefore, by summing \eqref{eq:bowtieeneg} and \eqref{eq:bowtieepos}, the coefficient of $e_{32}$ in $X_G(\bm x;q)$ is $q^2(1+q)^2$, as we saw in \eqref{eq:bowtiee}.
\end{example}

\begin{example}\label{ex:coefn}
Forest triples of $G$ of type $n$ consist of a single tree triple $\mathcal T=(T,\alpha,r)$, where $T$ is a decreasing spanning tree of $G$, $\alpha$ consists of the single part $n$, and $1\leq r\leq n$. Therefore, the coefficient of $e_n$ in $X_G(\bm x;q)$ is 
\begin{equation}\label{eq:coefen}
c_n(q)=[n]_q\sum_{T\text{ dec. sp. tree}}q^{\text{inv}_G(\text{list}(T))}.
\end{equation}
We will show in Lemma \ref{lem:sumovertreelists} that \eqref{eq:coefen} is equal to $[n]_q[b_2]_q[b_3]_q\cdots[b_n]_q$, where $b_i$ is the number of smaller neighbours of vertex $i$ in $G$. For example, the coefficient of $e_5$ in $X_G(\bm x;q)$ for the bowtie graph $G$ is $[5]_q[2]_q[2]_q$, as we saw in \eqref{eq:bowtiee}. This recovers the result \cite[Corollary 7.2]{chromposquasi}.
\end{example}

Before we prove Theorem \ref{thm:foresttriples}, we will use the example of paths to show how finding a \emph{sign-reversing involution} $\varphi$ on the set $\text{FT}(G)$ of forest triples of $G$ can be used to prove that $X_G(\bm x;q)$ is $e$-positive and give an explicit formula. The following result is well-known \cite[Section 5]{chromposquasi}, \cite[Proposition 5.3]{chromquasihessenberg}, but this proof technique is new.

\begin{proposition} \label{prop:paths} The path $P_n=([n],\{\{i,i+1\}:1\leq i\leq n-1\})$ has chromatic quasisymmetric function given by
\begin{align}\label{eq:path}
X_{P_n}(\bm x;q)&=\sum_{\alpha\vDash n}[\alpha_1]_q([\alpha_2]_q-1)\cdots([\alpha_m]_q-1)e_{\text{sort}(\alpha)}\\\nonumber&=\sum_{\alpha\vDash n}q^{m-1}[\alpha_1]_q[\alpha_2-1]_q\cdots[\alpha_m-1]_qe_{\text{sort}(\alpha)},
\end{align}
where we sum over integer compositions $\alpha=\alpha_1\cdots\alpha_m$ of size $n$. In particular, $X_{P_n}(\bm x;q)$ is $e$-positive.
\end{proposition}

\begin{example}
The chromatic quasisymmetric function $X_{P_6}(\bm x;q)$ with the compositions contributing to \eqref{eq:path}, is given by
\begin{align}\label{eq:P6example}
X_{P_6}(\bm x;q)&=q^2[2]_qe_{222}+q^2[2]_qe_{321}+q^2[2]_qe_{321}+q[3]_q[2]_qe_{33}\\\nonumber&\ \ \ \ \ \ \ \ \ \ \ \
\text{\tiny{\textcolor{red}{222}}} \ \ \ \ \ \ \ \ \ \ \ \ \text{\tiny{\textcolor{red}{123}}} \ \ \ \ \ \ \ \ \ \ \ \ \
\text{\tiny{\textcolor{red}{132}}} \ \ \ \ \ \ \ \ \ \ \ \ \ \ \
\text{\tiny{\textcolor{red}{33}}} \ \ \ \ \ \ \ \ \ \ \ \ \ \ \\\nonumber&+q[2]_q[3]_qe_{42}+q[4]_qe_{42}+q[4]_qe_{51}+[6]_qe_6.
\\\nonumber&\ \ \ \ \ \ \ \ \ \ \ \ \ \ \
\text{\tiny{\textcolor{red}{24}}} \ \ \ \ \ \ \ \ \ \ \
\text{\tiny{\textcolor{red}{42}}} \ \ \ \ \ \ \ \ \ \ \
\text{\tiny{\textcolor{red}{15}}} \ \ \ \ \ \ \ \ \ \
\text{\tiny{\textcolor{red}{6}}}
\end{align}
\end{example}

\begin{proof} [Proof of Proposition \ref{prop:paths}. ] Because decreasing subtrees of $P_n$ are simply paths from some vertex $i$ to some vertex $j>i$, which we will denote $P_{i\to j}$, a forest triple $\mathcal F\in \text{FT}(P_n)$ must be of the form 
\begin{equation}
\mathcal F=(\mathcal T_1=(P_{1\to i_2-1},\alpha^{(1)},r_1),\mathcal T_2=(P_{i_2\to i_3-1},\alpha^{(2)},r_2),\ldots,\mathcal T_m=(P_{i_m\to n},\alpha^{(m)},r_m)).
\end{equation}
Also note that we always have $\sigma=123\cdots n$, so $\text{inv}_{P_n}(\sigma)=0$ and $\text{weight}(\mathcal F)=\sum_{i=1}^m(r_i-1)$. We will now define a forest triple $\varphi(\mathcal F)$ as follows. Let $1\leq j\leq m$ be maximal, if one exists, such that either $\ell(\alpha^{(j)})\geq 2$, or we have $j\geq 2$, $\ell(\alpha^{(j)})=1$, and $r_j=1$. If $\ell(\alpha^{(j)})\geq 2$, then letting $\alpha^{(j)}\setminus\alpha^{(j)}_\ell$ denote the composition $\alpha^{(j)}$ with its last part $\alpha^{(j)}_\ell$ removed, we define $\varphi(\mathcal F)$ by replacing $\mathcal T_j$ by the two tree triples
\begin{equation}
\mathcal S_1=(P_{i_j\to i_{j+1}-1-\alpha^{(j)}_\ell},\alpha^{(j)}\setminus\alpha^{(j)}_\ell,r_j) \text{ and }\mathcal S_2=(P_{i_{j+1}-\alpha^{(j)}_\ell\to i_{j+1}-1},\alpha^{(j)}_\ell,1).
\end{equation}
On the other hand, if $j\geq 2$, $\ell(\alpha^{(j)})=1$, and $r_j=1$, then we define $\varphi(\mathcal F)$ by replacing $\mathcal T_{j-1}$ and $\mathcal T_j$ by the tree triple
\begin{equation}
\mathcal T=(P_{i_{j-1}\to i_{j+1}-1},\alpha^{(j-1)}\cdot\alpha^{(j)},r_{j-1}).
\end{equation}
If no such $j$ exists, we set $\varphi(\mathcal F)=\mathcal F$. Informally, we look at our tree triples from right to left, identifying the first one where either the composition has at least two parts, in which case we break the path into two, or where we can join with the previous tree triple to undo this process. \\

We now claim that $\varphi$ is an \emph{involution}, meaning that $\varphi(\varphi(\mathcal F))=\mathcal F$. This is clear if $\mathcal F$ is a \emph{fixed point}, meaning that $\varphi(\mathcal F)=\mathcal F$, otherwise by maximality of $j$, we have $\ell(\alpha^{(k)})=1$ and $r_k\geq 2$ for $j+1\leq k\leq m$, so applying $\varphi$ again will replace the desired tree triples. We also see that the map $\varphi$ preserves type because the parts of the compositions do not change and it preserves weight because the $r_i$ that appear or disappear are always equal to one. If $\mathcal F$ is a fixed point we have $\text{sign}(\mathcal F)=+1$ because every composition must have length one, and if $\mathcal F$ is not a fixed point then $\varphi$ is \emph{sign-reversing}, meaning that $\text{sign}(\varphi(\mathcal F))=-\text{sign}(\mathcal F)$, because the total number of tree triples changes by one.\\

In summary, this means that the map $\varphi$ is a type-and-weight-preserving bijection between the negatively-signed non-fixed forest triples and the positively-signed non-fixed forest triples of $P_n$. Therefore, \eqref{eq:ftformula} reduces to a sum over only the fixed forest triples. Such an $\mathcal F$ can be associated with a composition $\alpha\vDash n$ by reading the tree lengths from left to right, we will have $\text{type}(\mathcal F)=\text{sort}(\alpha)$, and we must have $r_i\geq 2$ for every $i\geq 2$, so we have as desired
\begin{align}
X_{P_n}(\bm x;q)&=\sum_{\mathcal F\text{ fixed }}q^{\text{weight}(\mathcal F)}e_{\text{type}(\mathcal F)}=\sum_{\alpha\vDash n}\sum_{\substack{1\leq r_i\leq\alpha_i\\r_i\geq 2\text{ if }i\geq 2}}q^{\sum_{i=1}^m(r_i-1)}e_{\text{sort}(\alpha)}\\\nonumber&=\sum_{\alpha\vDash n}[\alpha_1]_q([\alpha_2]_q-1)\cdots([\alpha_m]_q-1)e_{\text{sort}(\alpha)}.
\end{align}
\end{proof}

\begin{example} Figure \ref{fig:pathexample} shows all of the forest triples of $P_6$ of type $222$ paired under our sign-reversing involution $\varphi$. The compositions are not written but they are $2$, $22$, or $222$. We have indicated whether each $r$ is $1$ or $2$ by circling the $r$-th smallest vertex of the corresponding tree. Informally, we can join two paths if the path on the right has its leftmost vertex circled. The coefficient of $e_{222}$ in $X_{P_6}(\bm x;q)$ is $(q^2+q^3)$, corresponding to the two fixed points. 
\end{example}

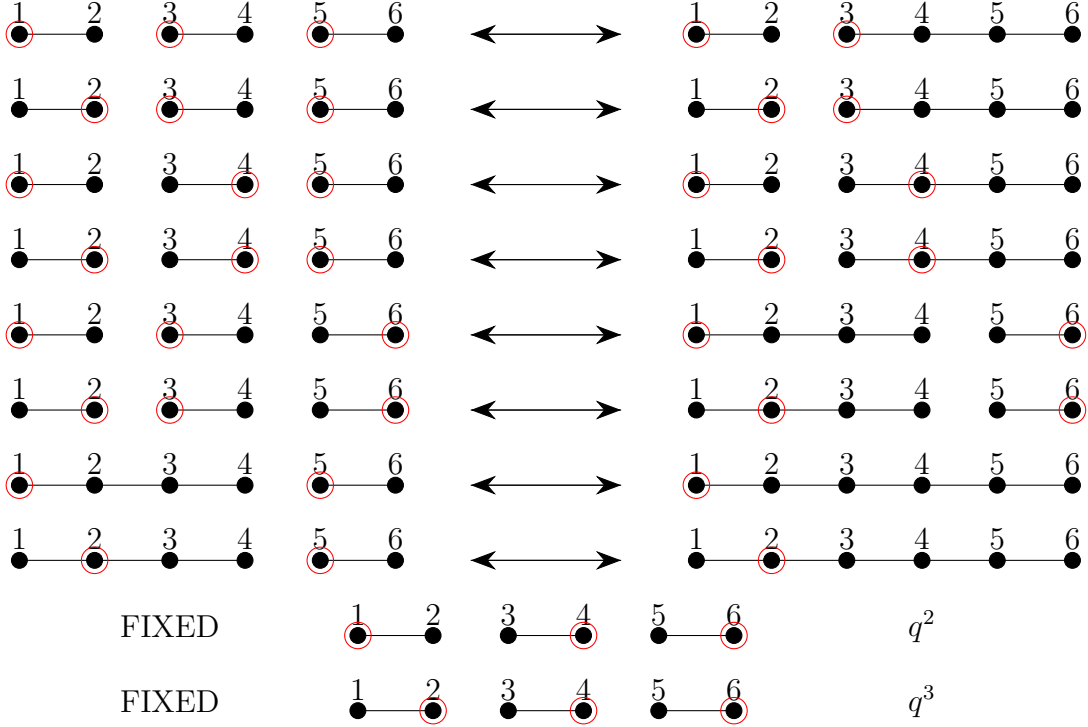
\begin{figure}\caption{\label{fig:pathexample} The forest triples of $P_6$ of type $222$ paired under $\varphi$}
\begin{tikzpicture}
\filldraw (1,0) circle (3pt) node[align=center,above] (1){1};
\filldraw (2,0) circle (3pt) node[align=center,above] (2){2};
\filldraw (3,0) circle (3pt) node[align=center,above] (3){3};
\filldraw (4,0) circle (3pt) node[align=center,above] (4){4};
\filldraw (5,0) circle (3pt) node[align=center,above] (5){5};
\filldraw (6,0) circle (3pt) node[align=center,above] (6){6};
\draw (1,0)--(2,0) (3,0)--(4,0) (5,0)--(6,0);
\draw [color=red](1,0) circle (5pt) node[align=center,above](){};
\draw [color=red](3,0) circle (5pt) node[align=center,above](){};
\draw [color=red](5,0) circle (5pt) node[align=center,above](){};
\draw [arrows = {-Stealth[scale=2]}] (7,0)--(9,0);
\draw [arrows = {-Stealth[scale=2]}] (9,0)--(7,0);
\filldraw (10,0) circle (3pt) node[align=center,above] (1){1};
\filldraw (11,0) circle (3pt) node[align=center,above] (2){2};
\filldraw (12,0) circle (3pt) node[align=center,above] (3){3};
\filldraw (13,0) circle (3pt) node[align=center,above] (4){4};
\filldraw (14,0) circle (3pt) node[align=center,above] (5){5};
\filldraw (15,0) circle (3pt) node[align=center,above] (6){6};
\draw (10,0)--(11,0) (12,0)--(13,0)--(14,0)--(15,0);
\draw [color=red](10,0) circle (5pt) node[align=center,above](){};
\draw [color=red](12,0) circle (5pt) node[align=center,above](){};

\filldraw (1,-1) circle (3pt) node[align=center,above] (1){1};
\filldraw (2,-1) circle (3pt) node[align=center,above] (2){2};
\filldraw (3,-1) circle (3pt) node[align=center,above] (3){3};
\filldraw (4,-1) circle (3pt) node[align=center,above] (4){4};
\filldraw (5,-1) circle (3pt) node[align=center,above] (5){5};
\filldraw (6,-1) circle (3pt) node[align=center,above] (6){6};
\draw (1,-1)--(2,-1) (3,-1)--(4,-1) (5,-1)--(6,-1);
\draw [color=red](2,-1) circle (5pt) node[align=center,above](){};
\draw [color=red](3,-1) circle (5pt) node[align=center,above](){};
\draw [color=red](5,-1) circle (5pt) node[align=center,above](){};
\draw [arrows = {-Stealth[scale=2]}] (7,-1)--(9,-1);
\draw [arrows = {-Stealth[scale=2]}] (9,-1)--(7,-1);
\filldraw (10,-1) circle (3pt) node[align=center,above] (1){1};
\filldraw (11,-1) circle (3pt) node[align=center,above] (2){2};
\filldraw (12,-1) circle (3pt) node[align=center,above] (3){3};
\filldraw (13,-1) circle (3pt) node[align=center,above] (4){4};
\filldraw (14,-1) circle (3pt) node[align=center,above] (5){5};
\filldraw (15,-1) circle (3pt) node[align=center,above] (6){6};
\draw (10,-1)--(11,-1) (12,-1)--(13,-1)--(14,-1)--(15,-1);
\draw [color=red](11,-1) circle (5pt) node[align=center,above](){};
\draw [color=red](12,-1) circle (5pt) node[align=center,above](){};

\filldraw (1,-2) circle (3pt) node[align=center,above] (1){1};
\filldraw (2,-2) circle (3pt) node[align=center,above] (2){2};
\filldraw (3,-2) circle (3pt) node[align=center,above] (3){3};
\filldraw (4,-2) circle (3pt) node[align=center,above] (4){4};
\filldraw (5,-2) circle (3pt) node[align=center,above] (5){5};
\filldraw (6,-2) circle (3pt) node[align=center,above] (6){6};
\draw (1,-2)--(2,-2) (3,-2)--(4,-2) (5,-2)--(6,-2);
\draw [color=red](1,-2) circle (5pt) node[align=center,above](){};
\draw [color=red](4,-2) circle (5pt) node[align=center,above](){};
\draw [color=red](5,-2) circle (5pt) node[align=center,above](){};
\draw [arrows = {-Stealth[scale=2]}] (7,-2)--(9,-2);
\draw [arrows = {-Stealth[scale=2]}] (9,-2)--(7,-2);
\filldraw (10,-2) circle (3pt) node[align=center,above] (1){1};
\filldraw (11,-2) circle (3pt) node[align=center,above] (2){2};
\filldraw (12,-2) circle (3pt) node[align=center,above] (3){3};
\filldraw (13,-2) circle (3pt) node[align=center,above] (4){4};
\filldraw (14,-2) circle (3pt) node[align=center,above] (5){5};
\filldraw (15,-2) circle (3pt) node[align=center,above] (6){6};
\draw (10,-2)--(11,-2) (12,-2)--(13,-2)--(14,-2)--(15,-2);
\draw [color=red](10,-2) circle (5pt) node[align=center,above](){};
\draw [color=red](13,-2) circle (5pt) node[align=center,above](){};

\filldraw (1,-3) circle (3pt) node[align=center,above] (1){1};
\filldraw (2,-3) circle (3pt) node[align=center,above] (2){2};
\filldraw (3,-3) circle (3pt) node[align=center,above] (3){3};
\filldraw (4,-3) circle (3pt) node[align=center,above] (4){4};
\filldraw (5,-3) circle (3pt) node[align=center,above] (5){5};
\filldraw (6,-3) circle (3pt) node[align=center,above] (6){6};
\draw (1,-3)--(2,-3) (3,-3)--(4,-3) (5,-3)--(6,-3);
\draw [color=red](2,-3) circle (5pt) node[align=center,above](){};
\draw [color=red](4,-3) circle (5pt) node[align=center,above](){};
\draw [color=red](5,-3) circle (5pt) node[align=center,above](){};
\draw [arrows = {-Stealth[scale=2]}] (7,-3)--(9,-3);
\draw [arrows = {-Stealth[scale=2]}] (9,-3)--(7,-3);
\filldraw (10,-3) circle (3pt) node[align=center,above] (1){1};
\filldraw (11,-3) circle (3pt) node[align=center,above] (2){2};
\filldraw (12,-3) circle (3pt) node[align=center,above] (3){3};
\filldraw (13,-3) circle (3pt) node[align=center,above] (4){4};
\filldraw (14,-3) circle (3pt) node[align=center,above] (5){5};
\filldraw (15,-3) circle (3pt) node[align=center,above] (6){6};
\draw (10,-3)--(11,-3) (12,-3)--(13,-3)--(14,-3)--(15,-3);
\draw [color=red](11,-3) circle (5pt) node[align=center,above](){};
\draw [color=red](13,-3) circle (5pt) node[align=center,above](){};

\filldraw (1,-4) circle (3pt) node[align=center,above] (1){1};
\filldraw (2,-4) circle (3pt) node[align=center,above] (2){2};
\filldraw (3,-4) circle (3pt) node[align=center,above] (3){3};
\filldraw (4,-4) circle (3pt) node[align=center,above] (4){4};
\filldraw (5,-4) circle (3pt) node[align=center,above] (5){5};
\filldraw (6,-4) circle (3pt) node[align=center,above] (6){6};
\draw (1,-4)--(2,-4) (3,-4)--(4,-4) (5,-4)--(6,-4);
\draw [color=red](1,-4) circle (5pt) node[align=center,above](){};
\draw [color=red](3,-4) circle (5pt) node[align=center,above](){};
\draw [color=red](6,-4) circle (5pt) node[align=center,above](){};
\draw [arrows = {-Stealth[scale=2]}] (7,-4)--(9,-4);
\draw [arrows = {-Stealth[scale=2]}] (9,-4)--(7,-4);
\filldraw (10,-4) circle (3pt) node[align=center,above] (1){1};
\filldraw (11,-4) circle (3pt) node[align=center,above] (2){2};
\filldraw (12,-4) circle (3pt) node[align=center,above] (3){3};
\filldraw (13,-4) circle (3pt) node[align=center,above] (4){4};
\filldraw (14,-4) circle (3pt) node[align=center,above] (5){5};
\filldraw (15,-4) circle (3pt) node[align=center,above] (6){6};
\draw (10,-4)--(11,-4)--(12,-4)--(13,-4) (14,-4)--(15,-4);
\draw [color=red](10,-4) circle (5pt) node[align=center,above](){};
\draw [color=red](15,-4) circle (5pt) node[align=center,above](){};

\filldraw (1,-5) circle (3pt) node[align=center,above] (1){1};
\filldraw (2,-5) circle (3pt) node[align=center,above] (2){2};
\filldraw (3,-5) circle (3pt) node[align=center,above] (3){3};
\filldraw (4,-5) circle (3pt) node[align=center,above] (4){4};
\filldraw (5,-5) circle (3pt) node[align=center,above] (5){5};
\filldraw (6,-5) circle (3pt) node[align=center,above] (6){6};
\draw (1,-5)--(2,-5) (3,-5)--(4,-5) (5,-5)--(6,-5);
\draw [color=red](2,-5) circle (5pt) node[align=center,above](){};
\draw [color=red](3,-5) circle (5pt) node[align=center,above](){};
\draw [color=red](6,-5) circle (5pt) node[align=center,above](){};
\draw [arrows = {-Stealth[scale=2]}] (7,-5)--(9,-5);
\draw [arrows = {-Stealth[scale=2]}] (9,-5)--(7,-5);
\filldraw (10,-5) circle (3pt) node[align=center,above] (1){1};
\filldraw (11,-5) circle (3pt) node[align=center,above] (2){2};
\filldraw (12,-5) circle (3pt) node[align=center,above] (3){3};
\filldraw (13,-5) circle (3pt) node[align=center,above] (4){4};
\filldraw (14,-5) circle (3pt) node[align=center,above] (5){5};
\filldraw (15,-5) circle (3pt) node[align=center,above] (6){6};
\draw (10,-5)--(11,-5)--(12,-5)--(13,-5) (14,-5)--(15,-5);
\draw [color=red](11,-5) circle (5pt) node[align=center,above](){};
\draw [color=red](15,-5) circle (5pt) node[align=center,above](){};

\filldraw (1,-6) circle (3pt) node[align=center,above] (1){1};
\filldraw (2,-6) circle (3pt) node[align=center,above] (2){2};
\filldraw (3,-6) circle (3pt) node[align=center,above] (3){3};
\filldraw (4,-6) circle (3pt) node[align=center,above] (4){4};
\filldraw (5,-6) circle (3pt) node[align=center,above] (5){5};
\filldraw (6,-6) circle (3pt) node[align=center,above] (6){6};
\draw (1,-6)--(2,-6)--(3,-6)--(4,-6) (5,-6)--(6,-6);
\draw [color=red](1,-6) circle (5pt) node[align=center,above](){};
\draw [color=red](5,-6) circle (5pt) node[align=center,above](){};
\draw [arrows = {-Stealth[scale=2]}] (7,-6)--(9,-6);
\draw [arrows = {-Stealth[scale=2]}] (9,-6)--(7,-6);
\filldraw (10,-6) circle (3pt) node[align=center,above] (1){1};
\filldraw (11,-6) circle (3pt) node[align=center,above] (2){2};
\filldraw (12,-6) circle (3pt) node[align=center,above] (3){3};
\filldraw (13,-6) circle (3pt) node[align=center,above] (4){4};
\filldraw (14,-6) circle (3pt) node[align=center,above] (5){5};
\filldraw (15,-6) circle (3pt) node[align=center,above] (6){6};
\draw (10,-6)--(11,-6)--(12,-6)--(13,-6)--(14,-6)--(15,-6);
\draw [color=red](10,-6) circle (5pt) node[align=center,above](){};

\filldraw (1,-7) circle (3pt) node[align=center,above] (1){1};
\filldraw (2,-7) circle (3pt) node[align=center,above] (2){2};
\filldraw (3,-7) circle (3pt) node[align=center,above] (3){3};
\filldraw (4,-7) circle (3pt) node[align=center,above] (4){4};
\filldraw (5,-7) circle (3pt) node[align=center,above] (5){5};
\filldraw (6,-7) circle (3pt) node[align=center,above] (6){6};
\draw (1,-7)--(2,-7)--(3,-7)--(4,-7) (5,-7)--(6,-7);
\draw [color=red](2,-7) circle (5pt) node[align=center,above](){};
\draw [color=red](5,-7) circle (5pt) node[align=center,above](){};
\draw [arrows = {-Stealth[scale=2]}] (7,-7)--(9,-7);
\draw [arrows = {-Stealth[scale=2]}] (9,-7)--(7,-7);
\filldraw (10,-7) circle (3pt) node[align=center,above] (1){1};
\filldraw (11,-7) circle (3pt) node[align=center,above] (2){2};
\filldraw (12,-7) circle (3pt) node[align=center,above] (3){3};
\filldraw (13,-7) circle (3pt) node[align=center,above] (4){4};
\filldraw (14,-7) circle (3pt) node[align=center,above] (5){5};
\filldraw (15,-7) circle (3pt) node[align=center,above] (6){6};
\draw (10,-7)--(11,-7)--(12,-7)--(13,-7)--(14,-7)--(15,-7);
\draw [color=red](11,-7) circle (5pt) node[align=center,above](){};

\draw (3,-7.875) node (){FIXED};
\draw (3,-8.875) node (){FIXED};

\draw (13,-7.875) node (){$q^2$};
\draw (13,-8.875) node (){$q^3$};

\filldraw (5.5,-8) circle (3pt) node[align=center,above] (1){1};
\filldraw (6.5,-8) circle (3pt) node[align=center,above] (2){2};
\filldraw (7.5,-8) circle (3pt) node[align=center,above] (3){3};
\filldraw (8.5,-8) circle (3pt) node[align=center,above] (4){4};
\filldraw (9.5,-8) circle (3pt) node[align=center,above] (5){5};
\filldraw (10.5,-8) circle (3pt) node[align=center,above] (6){6};
\draw (5.5,-8)--(6.5,-8) (7.5,-8)--(8.5,-8) (9.5,-8)--(10.5,-8);
\draw [color=red](5.5,-8) circle (5pt) node[align=center,above](){};
\draw [color=red](8.5,-8) circle (5pt) node[align=center,above](){};
\draw [color=red](10.5,-8) circle (5pt) node[align=center,above](){};

\filldraw (5.5,-9) circle (3pt) node[align=center,above] (1){1};
\filldraw (6.5,-9) circle (3pt) node[align=center,above] (2){2};
\filldraw (7.5,-9) circle (3pt) node[align=center,above] (3){3};
\filldraw (8.5,-9) circle (3pt) node[align=center,above] (4){4};
\filldraw (9.5,-9) circle (3pt) node[align=center,above] (5){5};
\filldraw (10.5,-9) circle (3pt) node[align=center,above] (6){6};
\draw (5.5,-9)--(6.5,-9) (7.5,-9)--(8.5,-9) (9.5,-9)--(10.5,-9);
\draw [color=red](6.5,-9) circle (5pt) node[align=center,above](){};
\draw [color=red](8.5,-9) circle (5pt) node[align=center,above](){};
\draw [color=red](10.5,-9) circle (5pt) node[align=center,above](){};
\end{tikzpicture}
\end{figure}

We now prepare to prove Theorem \ref{thm:foresttriples}. We begin by studying the permutation $\text{list}(T)$. We will then use Alexandersson and Sulzgruber's combinatorial $e$-expansion of shifted \emph{LLT polynomials}, which were originally introduced by Lascoux, Leclerc, and Thibon \cite{lltoriginal} and shown by Carlsson and Mellit to be related to chromatic quasisymmetric functions via an operation called \emph{plethystic substitution}. This relationship was also studied by Alexandersson and Panova \cite{lltchrom}.

\begin{definition}
Let $\sigma=\sigma_1\cdots\sigma_h$ be a permutation of a subset $A\subseteq[n]$. A \emph{descent} of $\sigma$ is a pair of adjacent entries $(\sigma_i,\sigma_{i+1})$ with $\sigma_i>\sigma_{i+1}$. A \emph{left-to-right (LR) maximum} of $\sigma$ is an entry $\sigma_j$ such that $\sigma_j>\sigma_i$ for every $i<j$. In other words, reading the entries of $\sigma$ from left to right, when we read the entry $\sigma_j$, it is the maximum entry read by that point. We say that $\sigma$ is a \emph{tree list} of $G$ if it satisfies the following two conditions.
\begin{enumerate}
\item (Descent condition) For every descent $\sigma_i>\sigma_{i+1}$ of $\sigma$, we have $\{\sigma_{i+1},\sigma_i\}\in E(G)$.
\item (LR maxima condition) Let $a_1<\cdots<a_s$ be the sequence of LR maxima of $\sigma$. Then for every $1\leq i\leq s-1$, we have $\{a_i,a_{i+1}\}\in E(G)$. 
\end{enumerate}
\end{definition}

\begin{lemma}\label{lem:treelistbijection}
The function $\text{list}$ is a bijection between decreasing subtrees of $G$ and tree lists of $G$.
\end{lemma}

\begin{proof}
Let $T$ be a decreasing subtree of $G$ and let $\sigma=\text{list}(T)$ have sequence of LR maxima $a_1<\cdots<a_s$. If $\sigma_j$ is not an LR maximum, then we define its \emph{parent} to be the rightmost entry $\sigma_i$ with $\sigma_i>\sigma_j$ and $i<j$. We claim that the unique larger neighbour of $\sigma_j$ in $T$ is its parent $\sigma_i$. To have read $\sigma_j$ after the larger entry $\sigma_i$, it must be adjacent to some $\sigma_{i'}$ for $i\leq i'<j$, which if $j=i+1$ can only be $\sigma_i$, while if $j\geq i+2$ can again only be $\sigma_i$ because every $\sigma_{i'}$ for $i<i'<j$ already has unique larger neighbour its parent by induction on $j-i$. We also claim that for every $1\leq i\leq s-1$, the unique larger neighbour of $a_i$ in $T$ is $a_{i+1}$. When we read $a_{i+1}$, it must be adjacent to a vertex read already, but the non-LR maxima already have unique larger neighbour their parent and the LR maxima $a_{i'}$ for $1\leq i'<i$ already have unique larger neighbour $a_{i'+1}$ by induction on $i$, so the only possibility is $a_i$.\\

Now $\sigma$ is a tree list of $G$ because if $\sigma_i>\sigma_{i+1}$, then we have shown that $\{\sigma_{i+1},\sigma_i\}\in E(T)$ so $\{\sigma_{i+1},\sigma_i\}\in E(G)$, and for every $1\leq i\leq s-1$, we have shown that $\{a_i,a_{i+1}\}\in E(T)$ so $\{a_i,a_{i+1}\}\in E(G)$. Conversely, given a tree list $\sigma$, we can recover $T$ because we have identified the unique larger neighbour of every vertex. The LR maxima condition ensures that the edges $\{a_i,a_{i+1}\}$ are present in $G$ and thus available for $T$, and if $\sigma_j$ has parent $\sigma_i$, then we have $\sigma_{i+1}\leq \sigma_j<\sigma_i$, so the descent condition ensures that $\{\sigma_{i+1},\sigma_i\}\in E(G)$ and \eqref{eq:uicondition} ensures that the edge $\{\sigma_j,\sigma_i\}$ is present in $G$ and thus available for $T$. 
\end{proof}


\begin{remark}
Tree lists $\sigma\in S_n$ of $G$ are also in bijection with \emph{acyclic orientations} of $G$ with unique sink $1$, given by orienting each edge toward the earlier entry of $\sigma$.  
\end{remark}

\begin{lemma}\label{lem:sumovertreelists}
Let $A\subseteq [n]$ be nonempty and for a vertex $i\in A$, let $b_i(A)$ denote the number of vertices $j\in A$ with $j<i$ and $\{j,i\}\in E(G)$. Then we have
\begin{equation}\label{eq:sumovertreelists}
\sum_{\sigma\in S_A: \ \sigma\text{ tree list of }G}q^{\text{inv}_G(\sigma)}=\prod_{i\in A: \ i>\min(A)}[b_i(A)]_q.
\end{equation}
\end{lemma}

\begin{proof} We use induction on $A$, noting that the statement holds if $|A|=1$ or if $b_{\max(A)}(A)=0$, so assume otherwise and let $A'=A\setminus\{\max(A)\}$. A tree list $\sigma\in S_A$ arises from a tree list $\sigma'\in S_{A'}$ precisely by inserting $\max(A)$ either at the end of $\sigma'$ or directly before one of its neighbours other than the one occurring first in $\sigma'$. If $\max(A)$ is inserted at the end, it creates no new $G$-inversions, while if it is inserted directly before the $j$-th neighbour from the right, where $1\leq j\leq b_{\max(A)}(A)-1$, then it creates $j$ new $G$-inversions, so we multiply by $[b_{\max(A)}(A)]_q$ when we go from $A'$ to $A$.

\end{proof}

We now introduce LLT polynomials and plethystic substitution. All of the properties we will need can be blackboxed. 

\begin{definition}
Let $G=([n],E)$ be a natural unit interval graph. The \emph{LLT polynomial} is
\begin{equation}
\text{LLT}_G(\bm x;q)=\sum_{\substack{\kappa:[n]\to\mathbb N\\\text{ not necessarily proper }}}q^{\text{asc}(\kappa)}x_{\kappa(1)}x_{\kappa(2)}\cdots x_{\kappa(n)}.
\end{equation}
\end{definition}

Note that we are summing over arbitrary colourings of $G$, not just the proper ones. Now these power series are related by the following operation.

\begin{definition}
Let $a(q)=a_0+a_1q+\cdots+a_jq^j$ be a polynomial. The \emph{plethystic substitution} $f\mapsto f[a(q)\bm x]$ is defined by setting $p_k[a(q)\bm x]=(a_0+a_1q^k+\cdots+a_jq^{jk})p_k$ for every $k$, where $p_k=\sum_ix_i^k$ is the \emph{power sum symmetric function}, and extending to an algebra homomorphism on symmetric functions.
\end{definition}

\begin{proposition} \label{prop:plethysmrelation}\cite[Proposition 3.5]{shuffle} The chromatic quasisymmetric function $X_G(\bm x;q)$ is related to the LLT polynomial $\text{LLT}_G(\bm x;q)$ via the equation
\begin{equation}\label{eq:plethysmrelation}
X_G(\bm x;q)=\frac{\text{LLT}_G[(q-1)\bm x]}{(q-1)^n}.
\end{equation}
\end{proposition}

Now Alexandersson conjectured \cite[Conjecture 10]{lltunicellepos} and proved with Sulzgruber \cite[Corollary 2.10]{lltcombe} the following combinatorial $e$-expansion of $\text{LLT}_G(\bm x;q+1)$. They formulated it in terms of the highest reachable vertex, but this is equivalent by Remark \ref{rem:reverse}.

\begin{definition}
Let $\theta\subseteq E(G)$. For $u\in [n]$, the \emph{lowest reachable vertex} $\text{lrv}_\theta(u)$ is the smallest $v\in [n]$ such that the subgraph $([n],\theta)$ has a decreasing path from $u$ to $v$. Defining $u\sim_\theta u'$ if $\text{lrv}_\theta(u)=\text{lrv}_\theta(u')$, we let $\pi(\theta)=\{[u_1]_\theta,\ldots,[u_m]_\theta\}$ be the equivalence classes of $\sim_\theta$ and we define $\lambda(\theta)$ to be the partition of $n$ determined by their sizes.
\end{definition}

\begin{theorem}\label{thm:lltcombe}
Let $G$ be a natural unit interval graph. The shifted LLT polynomial $\text{LLT}_G(\bm x;q+1)$ satisfies
\begin{equation}\label{eq:lltformula}
\text{LLT}_G(\bm x;q+1)=\sum_{\theta\subseteq E(G)}q^{|\theta|}e_{\lambda(\theta)}.
\end{equation}
\end{theorem}

Therefore, we get an expression for $X_G(\bm x;q)$ by replacing $q$ by $(q-1)$ and applying plethystic substitution to the right hand side of \eqref{eq:lltformula}. We will use the following identities.

\begin{lemma}\hspace{0pt}\label{lem:plethysmidentities}
\begin{enumerate}
\item We have $e_n[(q-1)\bm x]=\sum_{k=0}^ne_k[q\bm x]e_{n-k}[-\bm x]$.
\item We have $e_n[-\bm x]=(-1)^nh_n$, where $h_n=\sum_{i_1\leq \cdots\leq i_n}x_{i_1}\cdots x_{i_n}$ is the \emph{complete homogeneous symmetric function}.
\item We have $h_n=\sum_{\alpha\vDash n}(-1)^{n-\ell(\alpha)}e_{\text{sort}(\alpha)}$.
\item We have \begin{equation}
\frac{e_n[(q-1)\bm x]}{q-1}=\sum_{\alpha\vDash n}(-1)^{\ell(\alpha)-1}[\alpha_1]_qe_{\text{sort}(\alpha)}.
\end{equation}
\end{enumerate}
\end{lemma}

\begin{proof} The first three parts are \cite[Equation 1.63]{qtcat}, \cite[Page 448]{enum2}, and \cite[Theorem 2.18]{countingsym}. To prove the fourth, we use the others and we associate a composition of $(n-k)$ to a composition of $n$ by prepending a $k$ to find that 
\begin{align}
\frac{e_n[(q-1)\bm x]}{q-1}&=\frac 1{q-1}\sum_{k=0}^ne_k[q\bm x]e_{n-k}[-\bm x]=\frac 1{q-1}\sum_{k=0}^nq^ke_k(-1)^{n-k}h_{n-k}\\\nonumber&=\frac 1{q-1}\sum_{k=0}^nq^ke_k(-1)^{n-k}\left(\sum_{\alpha\vDash n-k}(-1)^{n-k-\ell(\alpha)}e_{\text{sort}(\alpha)}\right)\\\nonumber&=\frac 1{q-1}\left(\sum_{k=1}^n\sum_{\alpha\vDash n-k}q^k(-1)^{\ell(\alpha)}e_ke_{\text{sort}(\alpha)}-\sum_{\alpha\vDash n}(-1)^{\ell(\alpha)-1}e_{\text{sort}(\alpha)}\right)\\\nonumber&=\sum_{\alpha\vDash n}(-1)^{\ell(\alpha)-1}\frac{q^{\alpha_1}-1}{q-1}e_{\text{sort}(\alpha)}=\sum_{\alpha\vDash n}(-1)^{\ell(\alpha)-1}[\alpha_1]_qe_{\text{sort}(\alpha)}.
\end{align}
\end{proof}

Now we are ready to prove Theorem \ref{thm:foresttriples}.

\begin{proof}[Proof of Theorem \ref{thm:foresttriples}. ] We consider which $\theta\subseteq E(G)$ have $\pi(\theta)=\mathcal A$ for a fixed set partition $\mathcal A=\{A_1,\ldots,A_m\}$ of $[n]$ with $\min(A_1)<\cdots<\min(A_m)$. For each of the $(n-m)$ vertices $i$ that are not the minimum of their block, say $A_j$, $\theta$ must have at least one of the $b_i(A_j)$ edges of $B(i)=\{\{i',i\}\in E(G): i',i\in A_j, i'<i\}$ joining $i$ to a smaller vertex in the same block. Additionally, $\theta$ may optionally have any number of edges in \begin{equation}I(\mathcal A)=\{\{u,v\}\in E(G): u<v, u\in A_i, v\in A_j, i>j\}\end{equation} that join a smaller vertex in a later block to a larger vertex in an earlier block, because these edges do not affect $\pi(\theta)$. Therefore, using Lemma \ref{lem:treelistbijection} and Lemma \ref{lem:sumovertreelists}, we have
\begin{align}
\sum_{\substack{\theta\subseteq E(G)\\\pi(\theta)=\mathcal A}}\frac{(q-1)^{|\theta|}}{(q-1)^{n-m}}&=\left(\prod_{j=1}^m\sum_{\substack{i\in A_j\\i>\min(A_j)}}\frac 1{q-1}\sum_{\emptyset\neq S_i\subseteq B(i)}(q-1)^{|S_i|}\right)\sum_{S\subseteq I(\mathcal A)}(q-1)^{|S|}\\\nonumber&=q^{|I(A)|}\prod_{j=1}^m\prod_{\substack{i\in A_j\\i>\min(A_j)}}[b_i(A_j)]_q=q^{|I(A)|}\prod_{j=1}^m\sum_{\substack{\sigma^{(j)}\in S_{A_j}\\\sigma^{(j)}\text{ tree list of }G}}q^{\text{inv}_G(\sigma^{(j)})}\\\nonumber&=\sum_{\substack{T_1,\ldots,T_m\text{ dec. trees }\\V(T_j)=A_j}}q^{\text{inv}_G(\text{list}(T_1)\cdots\text{list}(T_m))},
\end{align}
where the factor of $q^{|I(\mathcal A)|}$ accounts for $G$-inversions of $\sigma=\text{list}(T_1)\cdots\text{list}(T_m)$ occurring between vertices in different blocks. Now using Proposition \ref{prop:plethysmrelation} and Theorem \ref{thm:lltcombe}, and separating $\theta\subseteq E(G)$ by the set partition $\pi(\theta)=\mathcal A=\{A_1,\ldots,A_m\}$, we have
\begin{align*}
X_G(\bm x;q)&=\sum_{\theta\subseteq E(G)}\frac{(q-1)^{|\theta|}e_{\lambda(\theta)}[(q-1)\bm x]}{(q-1)^n}=\sum_{\mathcal A}\left(\sum_{\substack{\theta\subseteq E(G)\\\pi(\theta)=\mathcal A}}\frac{(q-1)^{|\theta|}}{(q-1)^{n-m}}\right)\prod_{j=1}^m\frac{e_{|A_j|}[(q-1)\bm x]}{q-1}\\\nonumber&=\sum_{\mathcal A}\sum_{\substack{T_1,\ldots,T_m\text{ dec. trees }\\V(T_j)=A_j}}q^{\text{inv}_G(\sigma)}\prod_{j=1}^m\sum_{\alpha^{(j)}\vDash |A_j|}(-1)^{\ell(\alpha^{(j)})-1}[\alpha^{(j)}_1]_qe_{\text{sort}(\alpha^{(j)})}
\\\nonumber&=\sum_{\substack{T_1\ldots,T_m\text{ dec. trees }\\V(T_1)\sqcup\cdots\sqcup V(T_m)=[n]\\\alpha^{(1)}\vDash|V(T_1)|,\ldots,\alpha^{(m)}\vDash|V(T_m)|\\1\leq r_1\leq\alpha^{(1)}_1,\ldots,1\leq r_m\leq\alpha^{(m)}_1}}(-1)^{\sum_{j=1}^m(\ell(\alpha^{(j)})-1)}q^{\text{inv}_G(\sigma)}q^{\sum_{j=1}^m(r_j-1)}e_{\text{sort}(\alpha^{(1)}\cdots\alpha^{(m)})}\\\nonumber&=\sum_{\mathcal F\in\text{FT}(G)}\text{sign}(\mathcal F)q^{\text{weight}(\mathcal F)}e_{\text{type}(\mathcal F)}.
\end{align*}

\end{proof}

\section{Complete chains}\label{section:kchains}
In this section, we prove a combinatorial $e$-expansion of the chromatic quasisymmetric function of any $K$-chain. We begin by setting some notation.

\begin{definition}
We define the \emph{sum} of graphs $G_1=([n_1],E_1)$ and $G_2=([n_2],E_2)$ to be
\begin{equation}
G_1+G_2=([n_1+n_2-1],E_1\cup\{\{i+n_1-1,j+n_1-1\}: \ \{i,j\}\in E_2\}).
\end{equation}
Informally, vertex $n_1$ of $G_1$ and vertex $1$ of $G_2$ are glued together. For a composition $\gamma=\gamma_1\cdots\gamma_\ell$ with all parts at least $2$, we define the graph \begin{equation}K_\gamma=K_{\gamma_1}+\cdots+K_{\gamma_\ell},\end{equation} where $K_a$ is the complete graph on $[a]$. Graphs of the form $K_\gamma$ are called \emph{$K$-chains}. Note that $K$-chains are natural unit interval graphs. The bowtie graph from Figure \ref{fig:chromsymexamplebowtie} is $K_{33}$.
\end{definition}

This notation is due to Gebhard and Sagan, who proved \cite[Corollary 7.7]{chromnsym} that $X_{K_\gamma}(\bm x)$ is $e$-positive for every $\gamma$. We will prove that the chromatic \emph{quasisymmetric} function $X_{K_\gamma}(\bm x;q)$ is $e$-positive and $e$-unimodal by finding a sign-reversing involution on forest triples. We will in fact find an involution whose fixed points have all compositions of length one, and which preserves $\alpha^{(1)}_1(\mathcal F)$ and $r_1(\mathcal F)$, where $\mathcal T_1=(T_1,\alpha^{(1)}(\mathcal F),r_1(\mathcal F))$ is the tree triple of $\mathcal F$ with $1\in V(T_1)$, so we make the following definitions. Let $G$ be a natural unit interval graph.

\begin{definition}
A tree triple $\mathcal T=(T,\alpha,r)$ of $G$ is \emph{breakable} if $\ell(\alpha)\geq 2$. A forest triple $\mathcal F$ of $G$ is an \emph{atom} if all of its tree triples are not breakable.\end{definition}

Note that if $\mathcal T=(T,\alpha,r)$ is not breakable, then the composition $\alpha$ consists of the single part $|V(T)|$. In particular, if $\mathcal F$ is an atom, then $\text{sign}(\mathcal F)=(-1)^{\sum_{i=1}^m(1-1)}=+1$.


\begin{definition} A forest triple $\mathcal F$ of $G$ is \emph{simple} if $r_1(\mathcal F)=1$. Let $\text{FT}_{\text{simple}}(G)$ denote the set of simple forest triples of $G$. A \emph{nice involution} for $G$ is a function $\varphi:\text{FT}_{\text{simple}}(G)\to\text{FT}_{\text{simple}}(G)$ with the following properties.
\begin{enumerate}
\item $\varphi$ is an \emph{involution}, meaning that $\varphi(\varphi(\mathcal F))=\mathcal F$ for all $\mathcal F\in\text{FT}_{\text{simple}}(G)$. In other words, $\varphi$ is its own inverse, so in particular it is a bijection.
\item $\varphi$ preserves type, weight, and $\alpha^{(1)}_1$, which is the first part of the composition associated to the tree containing the vertex $1$.
\item If $\mathcal F$ is a \emph{fixed point} of $\varphi$, meaning that $\varphi(\mathcal F)=\mathcal F$, then $\mathcal F$ is an atom. 
\item $\varphi$ is \emph{sign-reversing}, meaning that if $\varphi(\mathcal F)\neq\mathcal F$, then $\text{sign}(\varphi(\mathcal F))=-\text{sign}(\mathcal F)$. 
\end{enumerate}
\end{definition}

Note that a forest triple $\mathcal F'\in\text{FT}(G)$ arises from some simple forest triple $\mathcal F$ by choosing some $1\leq r_1(\mathcal F')\leq\alpha^{(1)}_1(\mathcal F)$. Therefore, if we can find a nice involution $\varphi$ for $G$, then we have
\begin{equation}
X_G(\bm x;q)=\sum_{\mathcal F\in\text{FT}_{\text{simple}}(G), \varphi(\mathcal F)=\mathcal F}[\alpha^{(1)}_1(\mathcal F)]_q \ q^{\text{weight}(\mathcal F)}e_{\text{type}(\mathcal F)},
\end{equation} thus proving that $X_G(\bm x;q)$ is $e$-positive. The following Lemma will help us describe forest triples of a sum of graphs.

\begin{lemma}\label{lem:treelistofkchain}
Let $T$ be a decreasing subtree of $G$ that contains a cut vertex $c$. Then the tree list $\sigma=\text{list}(T)$ contains all of the vertices of $T$ at most $c$ in some order, followed by all of the vertices of $T$ strictly greater than $c$ in some order.
\end{lemma}

\begin{proof}
If not, then there is some $i$ with $\sigma_i>c$ and $\sigma_{i+1}\leq c$. By the LR maxima condition, the cut vertex $c$ must appear before $\sigma_i$, but this means that $\sigma_{i+1}<c$ and $\sigma$ does not satisfy the descent condition.
\end{proof}

We will want a standard way of taking a permutation $\sigma$ and writing a new permutation whose first entry is $\min(\sigma)$. The following choice is not the only possibility, but it will work.

\begin{definition}
Let $\sigma=\sigma_1\cdots\sigma_h$ be a permutation of some set $A=\{a_1<\cdots<a_h\}\subseteq [n]$. Define $\text{indstart}(\sigma)$ to be the integer $1\leq r\leq h$ such that $\sigma_1=a_r$ and define the permutation
\begin{equation}
\text{startmin}(\sigma)=w_1w_2\cdots w_{r-1}(\sigma),
\end{equation}
where $w_i$ is the transposition that switches the positions of the entries $a_i$ and $a_{i+1}$ in $\sigma$. Note that $\text{startmin}(\sigma)$ starts with its minimum element $a_1$, and if the entries in $A$ are all adjacent in $G$, then we have
\begin{equation}\label{eq:invrebalance}
\text{inv}_G(\sigma)=\text{inv}_G(\text{startmin}(\sigma))+(\text{indstart}(\sigma)-1)
\end{equation} because each of the $(\text{indstart}(\sigma)-1)$ transpositions that are applied decreases the number of $G$-inversions by exactly one. Conversely, given a permutation $\sigma=\sigma_1\cdots \sigma_h$ of $A$ with $\sigma_1=a_1$ and an integer $1\leq r\leq h$, define the permutation 
\begin{equation}\text{startr}(\sigma,r)=w_{r-1}\cdots w_2w_1(\sigma).
\end{equation}
Note that $\text{startr}(\sigma,r)$ starts with the element $a_r$ and  $\text{startmin}(\text{startr}(\sigma,\text{indstart}(\sigma)))=\sigma$.
\end{definition}

\begin{example}
For $\sigma=61743$, we have $\text{indstart}(\sigma)=4$ and $\text{startmin}(\sigma)=13764$.
\end{example}

We also define useful breaking and joining maps on tree triples. For convenience, we may refer to a tree $T$ and the permutation $\sigma=\text{list}(T)$ interchangeably.

\begin{definition} \label{def:easy}
Let $\mathcal T=(T,\alpha,r)$ be a breakable tree triple of $G$ with all vertices of $T$ adjacent in $G$, let $\alpha\setminus\alpha_\ell$ be the composition $\alpha$ with its last part $\alpha_\ell$ removed, and divide $\sigma=\text{list}(T)$ into its head and tail,
\begin{equation}\label{eq:headtail}
\sigma^{\text{head}}=\sigma_1 \ \cdots \ \sigma_{|V(T)|-\alpha_\ell}\text{ and }\sigma^{\text{tail}}=\sigma_{|V(T)|-\alpha_\ell+1} \ \cdots \ \sigma_{|V(T)|}.
\end{equation}
Then we define the pair of tree triples $\text{easybreak}(\mathcal T)=(\mathcal S_1,\mathcal S_2)$, where 
\begin{equation}
\mathcal S_1=(\sigma^{\text{head}},\alpha\setminus\alpha_\ell,r)\text{ and }\mathcal S_2=(\text{startmin}(\sigma^{\text{tail}}),\alpha_\ell,\text{indstart}(\sigma^{\text{tail}})).
\end{equation}
Conversely, given tree triples $\mathcal S_1=(S_1,\alpha^{(1)},r_1)$ and $\mathcal S_2=(S_2,\alpha^{(2)},r_2)$ with $\ell(\alpha^{(2)})=1$,  $\min(V(S_1))<\min(V(S_2))$, and $V(S_1)\sqcup V(S_2)$ all adjacent in $G$, we define the tree triple
\begin{equation}
\text{easyjoin}(\mathcal S_1,\mathcal S_2)=\mathcal T=(\text{list}(S_1)\cdot\text{startr}(\text{list}(S_2),r_2),\alpha^{(1)}\cdot\alpha^{(2)},r_1).
\end{equation}
Note that $\text{easybreak}(\mathcal T)=(\mathcal S_1,\mathcal S_2)$ if and only if $\text{easyjoin}(\mathcal S_1,\mathcal S_2)=\mathcal T$, and by \eqref{eq:invrebalance}, we have \begin{equation}\label{eq:inveasymap} \text{inv}_G(\text{list}(T))=\text{inv}_G(\text{list}(S_1)\cdot\text{list}(S_2))+(r_2-1).\end{equation}
\end{definition}

We now show how we can use the $\text{easybreak}$ and $\text{easyjoin}$ maps to find a nice involution for the complete graph $K_n$. We have already seen in Example \ref{ex:complete} that $X_{K_n}(\bm x;q)=[n]_q!e_n$, but this proof will be instructive for when we prove the main result of this section. Some examples are given in Figure \ref{fig:completeinvolution}.

\begin{figure}
\caption{\label{fig:completeinvolution} Some simple forest triples $\mathcal F$ of $K_6$ and $\varphi(\mathcal F)$}
\begin{tikzpicture}
\draw (1.25,0) node (){$K_6=$};
\filldraw (2,0) circle (3pt) node[align=center,above] (1){1};
\filldraw (2.5,0.866) circle (3pt) node[align=center,above] (2){2};
\filldraw (2.5,-0.866) circle (3pt) node[align=center,below] (3){3};
\filldraw (3.5,0.866) circle (3pt) node[align=center,above] (4){4};
\filldraw (3.5,-0.866) circle (3pt) node[align=center,below] (5){5};
\filldraw (4,0) circle (3pt) node[align=center,above] (6){6};

\draw (2,0)--(4,0) (2,0) -- (2.5,0.866) -- (3.5,0.866) -- (4,0) -- (3.5,-0.866) -- (2.5,-0.866) -- (2,0) (2,0) -- (3.5,-0.866) -- (3.5,0.866) -- (2,0) (4,0) -- (2.5,0.866) -- (2.5,-0.866) -- (4,0) (2.5,0.866) -- (3.5,-0.866) (2.5,-0.866) -- (3.5,0.866);
\end{tikzpicture}
\begin{equation*}
\begin{tabular}{|c|c|c|c|}
\hline
$\mathcal F$&$\varphi(\mathcal F)$&$\mathcal F$&$\varphi(\mathcal F)$\\\hline
$(\textcolor{red}{123}\textcolor{blue}{456},\textcolor{red}3\textcolor{blue}3,1)$&$(\textcolor{red}{123},\textcolor{red}3,1),(\textcolor{blue}{456},\textcolor{blue}3,\textcolor{teal}1)$&$(\textcolor{red}{1362}\textcolor{blue}{54},\textcolor{red}{22}\textcolor{blue}2,1)$&$(\textcolor{red}{1362},\textcolor{red}{22},1),(\textcolor{blue}{45},\textcolor{blue}2,\textcolor{teal}2)$\\\hline
$(\textcolor{red}{123}\textcolor{blue}{546},\textcolor{red}3\textcolor{blue}3,1)$&$(\textcolor{red}{123},\textcolor{red}3,1),(\textcolor{blue}{456},\textcolor{blue}3,\textcolor{teal}2)$&$(\textcolor{gray}{13},\textcolor{gray}{2},\textcolor{gray}1),(\textcolor{red}{26}\textcolor{blue}{45},\textcolor{red}2\textcolor{blue}2,2)$&$(\textcolor{gray}{13},\textcolor{gray}{2},\textcolor{gray}1),(\textcolor{red}{26},\textcolor{red}2,2),(\textcolor{blue}{45},\textcolor{blue}2,\textcolor{teal}1)$\\\hline
$(\textcolor{red}{123}\textcolor{blue}{654},\textcolor{red}3\textcolor{blue}3,1)$&$(\textcolor{red}{123},\textcolor{red}3,1),(\textcolor{blue}{465},\textcolor{blue}3,\textcolor{teal}3)$&$(\textcolor{gray}{15},\textcolor{gray}2,\textcolor{gray}1),(\textcolor{red}{24}\textcolor{blue}{63},\textcolor{red}2\textcolor{blue}2,1)$&$(\textcolor{gray}{15},\textcolor{gray}2,\textcolor{gray}1),(\textcolor{red}{24},\textcolor{red}{2},1),(\textcolor{blue}{36},\textcolor{blue}2,\textcolor{teal}2)$\\\hline
\end{tabular}
\end{equation*}
\end{figure}
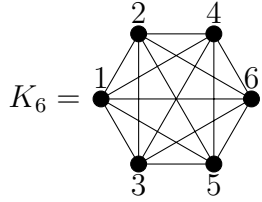

\begin{proposition}\label{prop:complete}
There is a nice involution $\varphi$ for the complete graph $K_n$. Furthermore, the fixed points of $\varphi$ are exactly
\begin{equation}
\text{Fix}(\varphi)=\{\mathcal F\in\text{FT}_{\text{simple}}(K_n): \ \mathcal F\text{ is an atom,} \ \mathcal F\text{ consists of a single tree triple}\}.
\end{equation} In particular, by Lemma \ref{lem:sumovertreelists}, the chromatic quasisymmetric function is
\begin{equation}
X_{K_n}(\bm x;q)=\sum_{\mathcal F\in\text{Fix}(\varphi)}[\alpha^{(1)}_1(\mathcal F)]_q \ q^{\text{weight}(\mathcal F)}e_{\text{type}(\mathcal F)}=[n]_q\sum_{\sigma\in S_n: \ \sigma_1=1}q^{\text{inv}_{K_n}(\sigma)}e_n=[n]_q!e_n.
\end{equation}
\end{proposition}

\begin{proof}
We build our nice involution $\varphi$ using the $\text{easybreak}$ and $\text{easyjoin}$ maps. Let \begin{equation}\mathcal F=(\mathcal T_1=(T_1,\alpha^{(1)},1),\ldots,\mathcal T_m=(T_m,\alpha^{(m)},r_m))\in\text{FT}_{\text{simple}}(K_n).\end{equation}
If $\mathcal T_m$ is breakable, with $\text{easybreak}(\mathcal T)=(\mathcal S_1,\mathcal S_2)$, then we define $\varphi(\mathcal F)$ by replacing $\mathcal T_m$ by $\mathcal S_1$ and $\mathcal S_2$, that is, $\varphi(\mathcal F)=(\mathcal T_1,\ldots,\mathcal T_{m-1},\mathcal S_1,\mathcal S_2)$. Note that by our definition in \eqref{eq:headtail}, the tree triples are in the correct order and weight is preserved by \eqref{eq:inveasymap}. We have that $\varphi$ reverses sign and preserves type and $\alpha^{(1)}_1$.\\

If $\mathcal T_m$ is not breakable and $m=1$, then $\mathcal F\in\text{Fix}(\varphi)$ and we perforce define $\varphi(\mathcal F)=\mathcal F$. If $\mathcal T_m$ is not breakable and $m\geq 2$, then we define $\varphi(\mathcal F)$ by replacing $\mathcal T_{m-1}$ and $\mathcal T_m$ by $\mathcal T=\text{easyjoin}(\mathcal T_{m-1},\mathcal T_m)$, that is,
$\varphi(\mathcal F)=(\mathcal T_1,\ldots,\mathcal T_{m-2},\mathcal T)$. Note that by our definition in \eqref{eq:headtail}, the tree triples are in the correct order and weight is preserved by \eqref{eq:inveasymap}. We have that $\varphi$ reverses sign and preserves type and $\alpha^{(1)}_1$. By construction, the map $\varphi$ is an involution with fixed points exactly $\text{Fix}(\varphi)$, as desired.

\end{proof}

We now define a notion of restricting a simple forest triple $\mathcal F$ of $G_1+G_2$ to one of $G_2$ and we define an invariant that describes how many ``segments'' of trees $\mathcal F$ has present in $G_1$. Some examples are given in Figure \ref{fig:restrictft}.

\begin{figure}
\caption{\label{fig:restrictft} Some simple forest triples $\mathcal F$ of $K_{64}$ and $\mathcal F\vert_{\geq 6}$}
\begin{tikzpicture}
\draw (1.25,0) node (){$K_{64}=$};
\filldraw (2,0) circle (3pt) node[align=center,above] (1){1};
\filldraw (2.5,0.866) circle (3pt) node[align=center,above] (2){2};
\filldraw (2.5,-0.866) circle (3pt) node[align=center,below] (3){3};
\filldraw (3.5,0.866) circle (3pt) node[align=center,above] (4){4};
\filldraw (3.5,-0.866) circle (3pt) node[align=center,below] (5){5};
\filldraw (4,0) circle (3pt) node[align=center,above] (6){6};
\filldraw (4.5,0.866) circle (3pt) node[align=center,above] (7){7};
\filldraw (4.5,-0.866) circle (3pt) node[align=center,below] (8){8};
\filldraw (5,0) circle (3pt) node[align=center,above] (9){9};

\draw (2,0)--(5,0) (2,0) -- (2.5,0.866) -- (3.5,0.866) -- (4,0) -- (3.5,-0.866) -- (2.5,-0.866) -- (2,0) (2,0) -- (3.5,-0.866) -- (3.5,0.866) -- (2,0) (4,0) -- (2.5,0.866) -- (2.5,-0.866) -- (4,0) (2.5,0.866) -- (3.5,-0.866) (2.5,-0.866) -- (3.5,0.866) (4,0) -- (4.5,0.866) -- (5,0) -- (4.5,-0.866) -- (4,0) (4.5,0.866) -- (4.5,-0.866);
\end{tikzpicture}
\begin{equation*}
\begin{tabular}{|c|c|c|c|}
\hline
$\mathcal F$&$\text{seg}_{[6]}(\mathcal F)$&$\mathcal T_i\vert_{\geq 6}$&$\mathcal F\vert_{\geq 6}$\\\hline
$(\textcolor{red}{12345}\textcolor{teal}6\textcolor{blue}{789},\textcolor{teal}9,1)$&$1$&$(\textcolor{teal}6\textcolor{blue}{789},\textcolor{teal}4,1)$&$(\textcolor{teal}1\textcolor{blue}{234},\textcolor{teal}4,1)$\\\hline
$(\textcolor{red}{123}\textcolor{teal}6\textcolor{red}{54}\textcolor{blue}{789},\textcolor{teal}7\textcolor{blue}2,1)$&$1$&$(\textcolor{teal}6\textcolor{blue}{789},\textcolor{teal}2\textcolor{blue}2,1)$&$(\textcolor{teal}1\textcolor{blue}{234},\textcolor{teal}2\textcolor{blue}2,1)$\\\hline
$(\textcolor{red}{13}\textcolor{teal}6\textcolor{red}{254}\textcolor{blue}8,\textcolor{red}5\textcolor{teal}2,1),(\textcolor{blue}{79},\textcolor{blue}2,2)$&$2$&$(\textcolor{teal}6\textcolor{blue}8,\textcolor{teal}2,1)$&$(\textcolor{teal}1\textcolor{blue}3,\textcolor{teal}2,1),(\textcolor{blue}{24},\textcolor{blue}2,2)$\\\hline
$(\textcolor{red}{143},\textcolor{red}3,1),(\textcolor{red}2\textcolor{teal}6\textcolor{red}{5}\textcolor{blue}{978},\textcolor{teal}3\textcolor{blue}3,3)$&$2$&$(\textcolor{teal}6\textcolor{blue}{978},\textcolor{teal}1\textcolor{blue}3,1)$&$(\textcolor{teal}1\textcolor{blue}{423},\textcolor{teal}1\textcolor{blue}3,1)$\\\hline
$(\textcolor{red}1\textcolor{teal}6\textcolor{red}{45}\textcolor{blue}{98},\textcolor{red}2\textcolor{teal}3\textcolor{blue}1,1),(\textcolor{red}{23},\textcolor{red}2,2),(\textcolor{blue}7,\textcolor{blue}1,1)$&$3$&$(\textcolor{teal}6\textcolor{blue}{98},\textcolor{teal}2\textcolor{blue}1,1)$&$(\textcolor{teal}1\textcolor{blue}{43},\textcolor{teal}2\textcolor{blue}1,1),(\textcolor{blue}2,\textcolor{blue}1,1)$\\\hline
\end{tabular}
\end{equation*}
\end{figure}
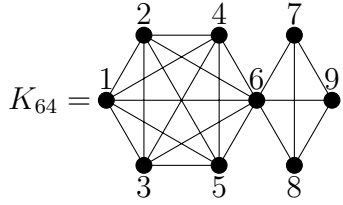

\begin{definition}
Let $G_1=([n_1],E_1)$ and $G_2=([n_2],E_2)$ be natural unit interval graphs and let $\mathcal F=(\mathcal T_1=(T_1,\alpha^{(1)},r_1=1),\ldots,\mathcal T_m=(T_m,\alpha^{(m)},r_m))$ be a simple forest triple of $G=G_1+G_2$. Let $i$ be such that $n_1\in V(T_i)$, let $\sigma=\text{list}(T_i)$, let $k=|V(T_i)\cap [n_1]|$, and noting that $k\leq |V(T_i)|=|\alpha^{(i)}|$, we define $\text{seg}_{[n_1]}(\mathcal T_i)$ to be the number $j$ such that
\begin{equation}
\alpha^{(i)}_1+\cdots+\alpha^{(i)}_{j-1}<k\leq\alpha^{(i)}_1+\cdots+\alpha^{(i)}_j.
\end{equation}
Letting $i'\geq i$ be maximal with $V(T_{i'})\cap [n_1]\neq\emptyset$, we say that the tree triples $\mathcal T_1,\ldots,\mathcal T_{i'}$ are \emph{present in $[n_1]$}, we define $\mathcal F\vert_{[n_1]}=(\mathcal T_1,\ldots,\mathcal T_{i'})$, and we define the integer
\begin{equation}\text{seg}_{[n_1]}(\mathcal F)=\text{seg}_{[n_1]}(\mathcal T_i)+\sum_{1\leq t\leq i', \ t\neq i}\ell(\alpha^{(t)}).\end{equation}
We now define the tree triple of $G$
\begin{equation}
\mathcal T_i\vert_{\geq n_1}=(n_1 \ \sigma_{k+1} \ \cdots \ \sigma_{\ell(\sigma)}, (\alpha^{(i)}_1+\cdots+\alpha^{(i)}_j-k+1)\alpha^{(i)}_{j+1}\cdots\alpha^{(i)}_\ell, 1)
\end{equation} and the simple forest triple of $G_2$
\begin{equation}
\mathcal F\vert_{\geq n_1}=(\mathcal T_i\vert_{\geq n_1}-(n_1-1),\mathcal T_{i'+1}-(n_1-1),\ldots,\mathcal T_m-(n_1-1)),
\end{equation}
where $\mathcal T-x$ denotes the tree triple obtained by subtracting $x$ from each entry of $\text{list}(T)$.\end{definition}

Note that by Lemma \ref{lem:treelistofkchain}, we have $\sigma_t\leq n_1$ for $1\leq t\leq k$ and $\sigma_t>n_1$ for $t>k$, so the permutation $n_1 \ \sigma_{k+1} \ \cdots \ \sigma_{\ell(\sigma)}$ starts with its minimum element and its descents and successive LR maxima appear in $\sigma$, so it is indeed a tree list of $G$ and $\mathcal F\vert_{\geq n_1}\in\text{FT}_{\text{simple}}(G_2)$.\\

We now prove the key result that will let us deduce $e$-positivity for $K$-chains by inductively adding complete graphs one at a time. The argument is identical if we add an \emph{almost-complete graph}, where a single edge is missing, so we prove both of these cases together. For $\epsilon\in\{0,1\}$, let $K_a^\epsilon$ denote the complete graph on $[a]$ if $\epsilon=0$ and the almost-complete graph on $[a]$, with the edge $(1,a)$ removed, if $\epsilon=1$. Note that every $\sigma$ with $\sigma_1=\min(\sigma)$ is a tree list for $K_a$, and is a tree list for $K_a^{(1)}$ unless $\sigma_1=1$ and $\sigma_2=a$.

\begin{theorem}\label{thm:key} Let $\varphi'$ be a nice involution for a natural unit interval graph $G'$. Then there is a nice involution $\varphi$ for $G=K_a^\epsilon+G'$. Furthermore, the fixed points of $\varphi$ are exactly
\begin{align}
\text{Fix}(\varphi)=\{\mathcal F\in\text{FT}&_{\text{simple}}(G): \ \mathcal F\text{ is an atom, }\mathcal F\vert_{\geq a}\text{ is fixed under }\varphi', \ \text{seg}_{[a]}(\mathcal F)\leq 2,\\\nonumber&\text{ if }1\notin V(T_i),\text{ then }r_i\geq a-\epsilon,\text{ if }1\in V(T_i),\text{ then }|V(T_i)|\geq a\},
\end{align}where $\mathcal T_i=(T_i,\alpha^{(i)},r_i)$ is the tree triple of $\mathcal F$ with $a\in V(T_i)$.
\end{theorem}


\begin{proof} We perforce define $\varphi(\mathcal F)=\mathcal F$ for $\mathcal F\in\text{Fix}(\varphi)$, so we now build a sign-reversing involution with no fixed points on $\text{FT}_{\text{simple}}(G)\setminus\text{Fix}(\varphi)$ using several cases. Let 
\begin{equation}\mathcal F=(\mathcal T_1=(T_1,\alpha^{(1)},r_1),\ldots,\mathcal T_m=(T_m,\alpha^{(m)},r_m))\in\text{FT}_{\text{simple}}(G)\setminus\text{Fix}(\varphi),\end{equation} let $i$ be such that $a\in V(T_i)$, and let $i'\geq i$ be maximal with $V(T_{i'})\cap[a]\neq\emptyset$. Illustrative examples for $G=K_{64}=K_6+K_4$ will be provided in the figures.\\

\textbf{Case 1: $\mathcal F\vert_{\geq a}$ is not fixed under $\varphi'$.} We first define a sign-reversing involution $\varphi\vert_{\geq a}$ with no fixed points on those $\mathcal F$ for which $\varphi'(\mathcal F\vert_{\geq a})\neq\mathcal F\vert_{\geq a}$. The idea will be to restrict $\mathcal F$ to $G'$, apply the map $\varphi'$, and reattach to the rest of $\mathcal F$. Let $\sigma=\text{list}(T_i)$, $k=|V(T_i)\cap[a]|$, and $j=\text{seg}_{[a]}(\mathcal T_i)$. Apply the map $\varphi'$ to get $\varphi'(\mathcal F\vert_{\geq a})=\mathcal F'=(\mathcal T_1'=(T_1',\alpha^{(1)'},1),\mathcal T_2',\ldots,\mathcal T_{m'}')$ and let $\sigma'=\text{list}(T_1')$. We now define the tree triple of $G$
\begin{equation}
\mathcal T_i\vee\mathcal T_1'=(\sigma_1 \ \cdots \ \sigma_k \ (\sigma'_2+a-1) \ \cdots \ (\sigma'_{|V(T_1')|}+a-1), \alpha^{(i)}_1\cdots\alpha^{(i)}_j\alpha^{(1)'}_2\cdots\alpha^{(1)'}_\ell, r_i)
\end{equation} and we define $\varphi\vert_{\geq a}(\mathcal F)=\mathcal F\vee\mathcal F'$, where 
\begin{equation}
\mathcal F\vee\mathcal F'=(\mathcal T_1,\ldots,\mathcal T_{i-1},\mathcal T_i\vee\mathcal T_1',\mathcal T_{i+1},\ldots,\mathcal T_{i'},\mathcal T_2'+(a-1),\ldots,\mathcal T'_{m'}+(a-1)).
\end{equation}
The permutation $\sigma_1 \cdots \sigma_k \ (\sigma'_2+a-1) \cdots (\sigma'_{|V(T_1')|}+a-1)$ starts with its minimum element and its descents and successive LR maxima are those of $\sigma$ or translated ones in $\sigma'$ so it is a tree list of $G$. Because $\varphi'$ preserves the first part of the first composition, we have \begin{equation}\alpha^{(i)}_1+\cdots+\alpha^{(i)}_j+\alpha^{(1)'}_2+\cdots+\alpha^{(1)'}_\ell=k-1+\alpha^{(1)}_1(\mathcal F\vert_{\geq a})+\alpha^{(1)'}_2+\cdots+\alpha^{(1')}_\ell=k-1+|V(T_1')|
\end{equation} and we have $r_i\leq \alpha^{(i)}_1$, so $\mathcal T_i\vee\mathcal T_1'$ is indeed a tree triple of $G$ and $\varphi\vert_{\geq a}(\mathcal F)$ is a simple forest triple of $G$. Informally, $\mathcal F\vee\mathcal F'$ replaces the part of $\mathcal F$ in $G'$ by $\mathcal F'$. By construction, we have $\mathcal F\vee(\mathcal F\vert_{\geq a})=\mathcal F$ and we have $(\mathcal F\vee\mathcal F')\vert_{\geq a}=\mathcal F'$, which by hypothesis is not fixed under $\varphi'$. The map $\varphi\vert_{\geq a}$ reverses sign, preserves type, weight, and $\alpha^{(1)}_1$, has no fixed points, and is an involution because
\begin{equation}
\varphi\vert_{\geq a}(\varphi\vert_{\geq a}(\mathcal F))=\mathcal F\vee(\varphi'(\varphi'(\mathcal F\vert_{\geq a})))=\mathcal F\vee(\mathcal F\vert_{\geq a})=\mathcal F.
\end{equation}

\begin{figure}
\caption{\label{fig:varphiinduct} Some simple forest triples $\mathcal F$ of $K_{64}$ and $\varphi\vert_{\geq 6}(\mathcal F)$}

\begin{tikzpicture}
\draw (1.25,0) node (){$K_{64}=$};
\filldraw (2,0) circle (3pt) node[align=center,above] (1){1};
\filldraw (2.5,0.866) circle (3pt) node[align=center,above] (2){2};
\filldraw (2.5,-0.866) circle (3pt) node[align=center,below] (3){3};
\filldraw (3.5,0.866) circle (3pt) node[align=center,above] (4){4};
\filldraw (3.5,-0.866) circle (3pt) node[align=center,below] (5){5};
\filldraw (4,0) circle (3pt) node[align=center,above] (6){6};
\filldraw (4.5,0.866) circle (3pt) node[align=center,above] (7){7};
\filldraw (4.5,-0.866) circle (3pt) node[align=center,below] (8){8};
\filldraw (5,0) circle (3pt) node[align=center,above] (9){9};

\draw (2,0)--(5,0) (2,0) -- (2.5,0.866) -- (3.5,0.866) -- (4,0) -- (3.5,-0.866) -- (2.5,-0.866) -- (2,0) (2,0) -- (3.5,-0.866) -- (3.5,0.866) -- (2,0) (4,0) -- (2.5,0.866) -- (2.5,-0.866) -- (4,0) (2.5,0.866) -- (3.5,-0.866) (2.5,-0.866) -- (3.5,0.866) (4,0) -- (4.5,0.866) -- (5,0) -- (4.5,-0.866) -- (4,0) (4.5,0.866) -- (4.5,-0.866);
\end{tikzpicture}
\begin{equation*}
\begin{tabular}{|c|c|c|c|}
\hline
$\mathcal F$&$\mathcal F\vert_{\geq 6}$&$\mathcal F'$&$ \varphi\vert_{\geq 6}(\mathcal F)=\mathcal F\vee\mathcal F'$\\\hline
$(\textcolor{red}{123}\textcolor{teal}6\textcolor{red}{54}\textcolor{blue}{789},\textcolor{teal}7\textcolor{blue}2,1)$&$(\textcolor{teal}1\textcolor{blue}{234},\textcolor{teal}2\textcolor{blue}2,1)$&$(\textcolor{teal}1\textcolor{blue}2,\textcolor{teal}2,1),(\textcolor{blue}{34},\textcolor{blue}2,1)$&$(\textcolor{red}{123}\textcolor{teal}6\textcolor{red}{54}\textcolor{blue}7,\textcolor{teal}7,1),(\textcolor{blue}{89},\textcolor{blue}2,1)$\\\hline
$(\textcolor{red}{13}\textcolor{teal}6\textcolor{red}{254}\textcolor{blue}{897},\textcolor{red}4\textcolor{teal}3\textcolor{blue}2,1)$&$(\textcolor{teal}1\textcolor{blue}{342},\textcolor{teal}2\textcolor{blue}2,1)$&$(\textcolor{teal}1\textcolor{blue}3,\textcolor{teal}2,1),(\textcolor{blue}{24},\textcolor{blue}2,2)$&$(\textcolor{red}{13}\textcolor{teal}6\textcolor{red}{254}\textcolor{blue}8,\textcolor{red}4\textcolor{teal}3,1),(\textcolor{blue}{79},\textcolor{blue}2,2)$\\\hline
$(\textcolor{red}{154}\textcolor{teal}6\textcolor{red}{23}\textcolor{blue}{978},\textcolor{red}5\textcolor{teal}2\textcolor{blue}2,1)$&$(\textcolor{teal}1\textcolor{blue}{423},\textcolor{teal}2\textcolor{blue}2,1)$&$(\textcolor{teal}1\textcolor{blue}4,\textcolor{teal}2,1),(\textcolor{blue}{23},\textcolor{blue}2,1)$&$(\textcolor{red}{154}\textcolor{teal}6\textcolor{red}{23}\textcolor{blue}9,\textcolor{red}5\textcolor{teal}2,1),(\textcolor{blue}{78},\textcolor{blue}2,1)$\\\hline
$(\textcolor{red}1\textcolor{teal}6\textcolor{red}{3452}\textcolor{blue}{987},\textcolor{red}{23}\textcolor{teal}2\textcolor{blue}2,1)$&$(\textcolor{teal}1\textcolor{blue}{432},\textcolor{teal}2\textcolor{blue}2,1)$&$(\textcolor{teal}1\textcolor{blue}4,\textcolor{teal}2,1),(\textcolor{blue}{23},\textcolor{blue}2,2)$&$(\textcolor{red}1\textcolor{teal}6\textcolor{red}{3452}\textcolor{blue}9,\textcolor{red}{23}\textcolor{teal}2,1),(\textcolor{blue}{78},\textcolor{blue}2,2)$\\\hline
\end{tabular}\end{equation*}
\end{figure}

\textbf{Case 2: $\mathcal F\vert_{\geq a}$ is fixed under $\varphi'$ and $\text{seg}_{[a]}(\mathcal F)\leq 2$.} First note that if $\text{seg}_{[a]}(\mathcal F)=1$, then $i=1$, $\mathcal F$ is an atom, $[a]\subseteq V(T_1)$, and $\mathcal F\in\text{Fix}(\varphi)$, so we must have $\text{seg}_{[a]}(\mathcal F)=2$. We now define a bijection $\varphi_2$ between those $\mathcal F$ with $\text{seg}_{[a]}(\mathcal T_i)=2$, for which the only tree triple present in $[a]$ is of the form $\mathcal T_i=\mathcal T_1=(T_1,\alpha^{(1)}=\alpha^{(1)}_1\alpha^{(1)}_2,1)$, and those with $\text{seg}_{[a]}(\mathcal T_i)=1$, which have two tree triples present in $[a]$, which are both not breakable. Then we get a sign-reversing involution by either applying $\varphi_2$ or its inverse.\\

More specifically, for each $k$ we define a bijection
\begin{align}
\text{break}_k:&\{T\text{ dec. subtree of }G: \ [a]\subseteq V(T), \ k<|V(T)|<k+a\}\\\nonumber\to&\{(S_1,S_2,r_2): \ S_1,S_2\text{ dec. subtrees of }G,\ [a]\subseteq V(S_1)\sqcup V(S_2), \ 1\in V(S_1), \\\nonumber& \ \ |V(S_1)|<a, \ 1\leq r_2\leq |V(S_2)|=k,\text{ if }a\in V(S_2),\text{ then } r_2\leq a-\epsilon-1\},
\end{align} breaking a tree $T$ into $S_1$ and $S_2$, with $1\in V(S_1)$ and $|V(S_2)|=k$, such that \begin{equation}\label{eq:2segmentsinv} \text{inv}_G(\text{list}(T))=\text{inv}_G(\text{list}(S_1)\cdot\text{list}(S_2))+(r_2-1).\end{equation} 
Then we define $\varphi_2(\mathcal F)$ by taking $\text{break}_{\alpha^{(1)}_2}(T_1)=(S_1,S_2,r_2)$ and replacing $\mathcal T_1$ by the tree triples \begin{equation}\mathcal S_1=(S_1,\alpha^{(1)}_1,1)\text{ and }\mathcal S_2=(S_2,\alpha^{(1)}_2,r_2).\end{equation}
By construction, $\varphi_2$ will reverse sign, will preserve type and $\alpha^{(1)}_1$, will preserve weight by \eqref{eq:2segmentsinv}, and will have no fixed points. Let $T$ be a decreasing subtree of $G$ with $[a]\subseteq V(T)$ and $k<|V(T)|<k+a$, let $\sigma=\text{list}(T)$, and let $2\leq j\leq |V(T)|$ be such that $\sigma_j=a$. Note that if $j=2$, then we must have $\epsilon=0$. We will have three cases, depending on the position $j$ of the entry $a$ in $\sigma$. \\

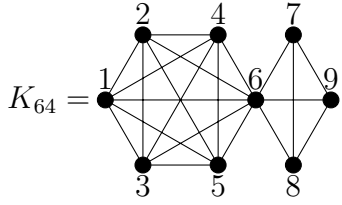
\begin{figure}
\caption{\label{fig:varphi2} Some simple forest triples $\mathcal F$ of $K_{64}$ and $\varphi_2(\mathcal F)$}

\begin{tikzpicture}
\draw (1.25,0) node (){$K_{64}=$};
\filldraw (2,0) circle (3pt) node[align=center,above] (1){1};
\filldraw (2.5,0.866) circle (3pt) node[align=center,above] (2){2};
\filldraw (2.5,-0.866) circle (3pt) node[align=center,below] (3){3};
\filldraw (3.5,0.866) circle (3pt) node[align=center,above] (4){4};
\filldraw (3.5,-0.866) circle (3pt) node[align=center,below] (5){5};
\filldraw (4,0) circle (3pt) node[align=center,above] (6){6};
\filldraw (4.5,0.866) circle (3pt) node[align=center,above] (7){7};
\filldraw (4.5,-0.866) circle (3pt) node[align=center,below] (8){8};
\filldraw (5,0) circle (3pt) node[align=center,above] (9){9};

\draw (2,0)--(5,0) (2,0) -- (2.5,0.866) -- (3.5,0.866) -- (4,0) -- (3.5,-0.866) -- (2.5,-0.866) -- (2,0) (2,0) -- (3.5,-0.866) -- (3.5,0.866) -- (2,0) (4,0) -- (2.5,0.866) -- (2.5,-0.866) -- (4,0) (2.5,0.866) -- (3.5,-0.866) (2.5,-0.866) -- (3.5,0.866) (4,0) -- (4.5,0.866) -- (5,0) -- (4.5,-0.866) -- (4,0) (4.5,0.866) -- (4.5,-0.866);
\end{tikzpicture}
\begin{equation*}
\begin{tabular}{|c|c|c|c|}
\hline$\mathcal F$&$\varphi_2(\mathcal F)$&$\mathcal F$&$\varphi_2(\mathcal F)$\\\hline
$(12345\textcolor{teal}6\textcolor{gray}{789},\textcolor{red}3\textcolor{blue}6,1)$&$(\textcolor{red}{123},\textcolor{red}3,1),(\textcolor{blue}{456789},\textcolor{blue}6,\textcolor{teal}1)$&$(12345\textcolor{teal}6\textcolor{gray}{789},\textcolor{red}5\textcolor{blue}4,1)$&$(\textcolor{red}{12345},\textcolor{red}5,1),(\textcolor{blue}{6789},\textcolor{blue}4,\textcolor{teal}1)$\\\hline
$(1324\textcolor{teal}65\textcolor{gray}{789},\textcolor{red}3\textcolor{blue}6,1)$&$(\textcolor{red}{132},\textcolor{red}3,1),(\textcolor{blue}{465789},\textcolor{blue}6,\textcolor{teal}2)$&$(15\textcolor{teal}6324\textcolor{gray}{789},\textcolor{red}5\textcolor{blue}4,1)$&$(\textcolor{red}{15324},\textcolor{red}5,1),(\textcolor{blue}{6789},\textcolor{blue}4,\textcolor{teal}4)$\\\hline
$(123\textcolor{teal}654\textcolor{gray}{789},\textcolor{red}3\textcolor{blue}6,1)$&$(\textcolor{red}{123},\textcolor{red}3,1),(\textcolor{blue}{465789},\textcolor{blue}6,\textcolor{teal}3)$&$(1\textcolor{teal}6\textcolor{orange}2\textcolor{blue}{345}\textcolor{gray}{789},\textcolor{red}5\textcolor{blue}4,1)$&$(\textcolor{red}{16789},\textcolor{red}5,1),(\textcolor{blue}{2345},\textcolor{blue}4,\textcolor{orange}1)$\\\hline
$(14\textcolor{teal}6532\textcolor{gray}{789},\textcolor{red}3\textcolor{blue}6,1)$&$(\textcolor{red}{145},\textcolor{red}3,1),(\textcolor{blue}{263789},\textcolor{blue}6,\textcolor{teal}4)$&$(1\textcolor{teal}6\textcolor{orange}3\textcolor{blue}{425}\textcolor{gray}{789},\textcolor{red}5\textcolor{blue}4,1)$&$(\textcolor{red}{16789},\textcolor{red}5,1),(\textcolor{blue}{2435},\textcolor{blue}4,\textcolor{orange}2)$\\\hline
$(1\textcolor{teal}65234\textcolor{gray}{789},\textcolor{red}3\textcolor{blue}6,1)$&$(\textcolor{red}{152},\textcolor{red}3,1),(\textcolor{blue}{346789},\textcolor{blue}6,\textcolor{teal}5)$&$(1\textcolor{teal}6\textcolor{orange}5\textcolor{blue}{432}\textcolor{gray}{789},\textcolor{red}5\textcolor{blue}4,1)$&$(\textcolor{red}{16789},\textcolor{red}5,1),(\textcolor{blue}{2543},\textcolor{blue}4,\textcolor{orange}4)$\\\hline
\end{tabular}
\end{equation*}
\end{figure}

If $j\geq|V(T)|-k+1$, then as in Definition \ref{def:easy}, we divide $\sigma$ into its head and tail,
\begin{equation}
\sigma^{\text{head}}=\sigma_1 \ \cdots \ \sigma_{|V(T)|-k}\text{ and }\sigma^{\text{tail}}=\sigma_{|V(T)|-k+1} \ \cdots \ \sigma_{|V(T)|},
\end{equation} and we define $(S_1,S_2,r_2)$ by 
\begin{equation}
\text{list}(S_1)=\sigma^{\text{head}}, \ \text{list}(S_2)=\text{startmin}(\sigma^{\text{tail}}),\text{ and }r_2=\text{indstart}(\sigma^{\text{tail}}).
\end{equation}
Because $|V(T)|-k<a$ and by Lemma \ref{lem:treelistofkchain}, $\sigma^{\text{head}}$ has all entries in $[a-1]$ so is a tree list of $G$, we have $\min(\sigma^{\text{tail}})\leq\sigma^{\text{tail}}_1\leq a$, so the only transpositions applied to $\sigma^{\text{tail}}$ switch entries at most $a$, which are all adjacent, and its descents and successive LR maxima appear in the tree list $\sigma$, so $\text{startmin}(\sigma^{\text{tail}})$ is a tree list of $G$, and \eqref{eq:2segmentsinv} follows from \eqref{eq:invrebalance}. Also note that $a\in V(S_2)$ and we have
\begin{equation}\label{eq:case21} 1\leq r_2\leq |V(S_2)\cap [a]|=a-(|V(T)|-k)\leq a-1\end{equation} with equality only if $k=|V(T)|-1$ and $\sigma_2=a$, meaning that $\epsilon=0$, so $r_2\leq a-\epsilon-1$. \\

If $a-k+1\leq j\leq|V(T)|-k$, then the entry $a$ would appear in $\sigma^{\text{head}}$ and $\sigma^{\text{tail}}$ would have entries less than $a$ and greater than $a$ and therefore would have no hope of satisfying the LR maximum condition. Therefore, we move the entry $a$ from position $j$ to the start of the tail, at position $|V(T)|-k+1$, and define 
\begin{equation}\sigma^{\text{head}'}=\sigma_1 \ \cdots \ \sigma_{j-1} \ \sigma_{j+1} \ \cdots \ \sigma_{|V(T)|-k+1}\text{ and }\sigma^{\text{tail}'}=a \ \sigma_{|V(T)|-k+2} \ \cdots \ \sigma_{\ell(\sigma)},
\end{equation} and we define $(S_1,S_2,r_2)$ by
\begin{equation} \text{list}(S_1)=\sigma^{\text{head}'}, \text{list}(S_2)=\text{startmin}(\sigma^{\text{tail}'}),\text{ and }r_2=a-j+1.\end{equation}
Informally, $r_2$ records the original position of the entry $a$. Again, $\sigma^{\text{head}'}$ has all entries in $[a-1]$ and the descents and successive LR maxima of $\text{startmin}(\sigma^{\text{tail}'})$ are either in $\{2,\ldots,a\}$ or appear in $\sigma$ so these are tree lists of $G$. Moving the entry $a$ from position $j$ to position $(|V(T)|-k+1)$ created $(|V(T)|-k+1-j)$ new $G$-inversions, and then we applied \begin{equation}\text{indstart}(\sigma^{\text{tail}'})-1=a-\ell(\sigma^{\text{head}'})-1=a-(|V(T)|-k)-1\end{equation} transpositions, so \eqref{eq:2segmentsinv} holds. Also note that $a\in V(S_2)$ and we have
\begin{equation}\label{eq:case22} a-(|V(T)|-k)+1\leq r_2\leq a-\epsilon-1\end{equation} because $j\geq 2$ with equality only if $\epsilon=0$.\\

If $j\leq a-k$, then we remove $k$ entries from the inside of $\sigma$ and define
\begin{equation}
\sigma^{\text{in}}=\sigma_{a-k+1} \ \cdots \ \sigma_a\text{ and }\sigma^{\text{out}}=\sigma_1 \ \cdots \ \sigma_{a-k} \ \sigma_{a+1} \ \cdots \ \sigma_{|V(T)|},
\end{equation} and we define $(S_1,S_2,r_2)$ by 
\begin{equation}
\text{list}(S_1)=\sigma^{\text{out}}, \ \text{list}(S_2)=\text{startmin}(\sigma^{\text{in}}),\text{ and }r_2=\text{indstart}(\sigma^{\text{in}}).
\end{equation}
Because $j\leq a-k$, the entry $a$ appears in $\sigma^{\text{out}}$ so its descents and successive LR maxima appear in $\sigma$ and $\text{startmin}(\sigma^{\text{in}})$ has all entries in $[a-1]$ so these are tree lists of $G$, and we have $r_2\leq\ell(\sigma^{\text{in}})=k$. Moving the string $\sigma^{\text{in}}$ does not affect the number of $G$-inversions because its entries are less than $a$, whereas every $\sigma_t>a$ for $t>a$ by Lemma \ref{lem:treelistofkchain}, so again \eqref{eq:2segmentsinv} follows from \eqref{eq:invrebalance}. Also note that the entries $\sigma_1$ and $\sigma_2$ are unaffected and $a\in V(S_1)$.\\

It remains to show that the map $\text{break}_k$ is a bijection. The idea is that we can determine which case of the map was applied by whether $a\in V(S_2)$ and from $r_2$, using \eqref{eq:case21} and \eqref{eq:case22}, so we can recover $\text{list}(T)$. If $a\in V(S_2)$ and $r_2\leq a-(|V(S_1)\sqcup V(S_2)|-k)$, then we must have been in the first case where $j\geq |V(T)|-k+1$ and the inverse map is given by \begin{equation}\text{list}(T)=\text{list}(S_1)\cdot\text{startr}(\text{list}(S_2),r_2).\end{equation}
Note that because $r_2\leq a-\epsilon-1$, the $r_2$-th smallest entry of $S_2$ can only be an $a$ if $\epsilon=0$, so $\text{list}(T)$ is indeed a tree list of $G$. If $a\in V(S_2)$ and $r_2\geq a-(|V(S_1)\sqcup V(S_2)|-k)+1$, then we must have been in the second case where $a-k+1\leq j\leq |V(T)|-k$ and the inverse map is given by concatenating the permutations $\tau^{(1)}=\text{list}(S_1)$ and $\tau^{(2)}=\text{startr}(\text{list}(S_2,a-|V(S_1)|))$ and then returning the entry $a$ from the start of $\tau^{(2)}$ to position $j=a-r_2+1$, that is, 
\begin{equation}
\text{list}(T)=\tau^{(1)}_1 \ \cdots \ \tau^{(1)}_{a-r_2} \ a \ \tau^{(1)}_{a-r_2+2} \ \cdots \ \tau^{(1)}_{|V(S_1)|} \ \tau^{(2)}_2 \ \cdots \ \tau^{(2)}_{k-1}. 
\end{equation} Note that because $r_2\leq a-\epsilon-1$, the entry $a$ can only move to position $j=2$ if $\epsilon=0$, so $\text{list}(T)$ is indeed a tree list of $G$. Finally, if $a\in V(S_1)$, then we must have been in the third case where $j\leq a-k$ and the inverse map is given by inserting $\tau^{(2)}=\text{startr}(\text{list}(S_2,r_2)$ into $\tau^{(1)}=\text{list}(S_1)$, specifically,
\begin{equation}
\text{list}(T)=\tau^{(1)}_1 \ \cdots \ \tau^{(1)}_{a-|V(S_1)|} \ \tau^{(2)}_1 \ \cdots \ \tau^{(2)}_{k} \ \tau^{(1)}_{a-|V(S_1)|+1} \ \cdots \ \tau^{(1)}_{|V(S_1)|}.
\end{equation}
These constructions are inverse to each other, so we indeed have a sign-reversing involution.\\

\textbf{Case 3: $\mathcal F\vert_{\geq a}$ is fixed under $\varphi'$, $\text{seg}_{[a]}(\mathcal F)\geq 3$, and $\text{seg}_{[a]}(\mathcal T_i)\leq\text{seg}_{[a]}(\mathcal F)-2$. } We now define a sign-reversing involution $\varphi_{\subseteq[a-1]}$ on these $\mathcal F$ by applying the $\text{easybreak}$ and $\text{easyjoin}$ maps used in the proof of Proposition \ref{prop:complete}, but now we carefully work around $\mathcal T_i$. Note that the tree triples present in $[a]$ other than $\mathcal T_i$ have all vertices in $[a-1]$. Because \begin{equation}\text{seg}_{[a]}(\mathcal F)-\text{seg}_{[a]}(\mathcal T_i)=\sum_{1\leq t\leq i', t\neq i}\ell(\alpha^{(t)})\geq 2,\end{equation} either some such tree triple is breakable or there are at least two such tree triples.\\

First suppose that $\sum_{i+1\leq t\leq i'}\ell(\alpha^{(t)})\geq 2$, meaning that $i'\geq i+1$ and either $\mathcal T_{i'}$ is breakable, or $\mathcal T_{i'}$ is not breakable and in fact $i'\geq i+2$. Informally, this means there are at least two segments of trees occurring after $T_i$. If $\mathcal T_{i'}$ is breakable, then we define $\varphi_{\subseteq[a-1]}(\mathcal F)$ by replacing $\mathcal T_{i'}$ by the two tree triples of $\text{easybreak}(\mathcal T_{i'})$, and if $\mathcal T_{i'}$ is not breakable, then we define $\varphi_{\subseteq[a-1]}(\mathcal F)$ by replacing $\mathcal T_{i'-1}$ and $\mathcal T_{i'}$ by $\text{easyjoin}(\mathcal T_{i'-1},\mathcal T_{i'})$. The map $\varphi_{\subseteq[a-1]}$ is a sign-reversing involution on these $\mathcal F$, it preserves type and $\alpha^{(1)}_1$, and it preserves weight by \eqref{eq:inveasymap} and because it preserves the order of the trees. \\

Similarly, suppose that $i'=i$, and either $\mathcal T_{i-1}$ is not breakable, or $\mathcal T_{i-1}$ is breakable with $\text{easybreak}(\mathcal T_{i-1})=(\mathcal S_1,\mathcal S_2=(S_2,\beta^{(2)},r'))$, and $\min(V(S_2))<\min(V(T_i))$. Then we again define $\varphi_{\subseteq[a-1]}(\mathcal F)$ by either replacing $\mathcal T_{i-1}$ by the two tree triples of $\text{easybreak}(\mathcal T_{i-1})$ if $\mathcal T_{i-1}$ is breakable, or by replacing $\mathcal T_{i-2}$ and $\mathcal T_{i-1}$ by $\text{easyjoin}(\mathcal T_{i-2},\mathcal T_{i-1})$ otherwise. The map $\varphi_{\subseteq[a-1]}$ is a sign-reversing involution on these $\mathcal F$, it preserves type and $\alpha^{(1)}_1$, and it preserves weight by \eqref{eq:inveasymap} and because it preserves the order of the trees.\\

\begin{figure}
\caption{\label{fig:varphieasy} Some simple forest triples $\mathcal F$ of $K_{64}$ and $\varphi_{\subseteq[5]}(\mathcal F)$}

\begin{tikzpicture}
\draw (1.25,0) node (){$K_{64}=$};
\filldraw (2,0) circle (3pt) node[align=center,above] (1){1};
\filldraw (2.5,0.866) circle (3pt) node[align=center,above] (2){2};
\filldraw (2.5,-0.866) circle (3pt) node[align=center,below] (3){3};
\filldraw (3.5,0.866) circle (3pt) node[align=center,above] (4){4};
\filldraw (3.5,-0.866) circle (3pt) node[align=center,below] (5){5};
\filldraw (4,0) circle (3pt) node[align=center,above] (6){6};
\filldraw (4.5,0.866) circle (3pt) node[align=center,above] (7){7};
\filldraw (4.5,-0.866) circle (3pt) node[align=center,below] (8){8};
\filldraw (5,0) circle (3pt) node[align=center,above] (9){9};

\draw (2,0)--(5,0) (2,0) -- (2.5,0.866) -- (3.5,0.866) -- (4,0) -- (3.5,-0.866) -- (2.5,-0.866) -- (2,0) (2,0) -- (3.5,-0.866) -- (3.5,0.866) -- (2,0) (4,0) -- (2.5,0.866) -- (2.5,-0.866) -- (4,0) (2.5,0.866) -- (3.5,-0.866) (2.5,-0.866) -- (3.5,0.866) (4,0) -- (4.5,0.866) -- (5,0) -- (4.5,-0.866) -- (4,0) (4.5,0.866) -- (4.5,-0.866);
\end{tikzpicture}
\begin{equation*}
\begin{tabular}{|c|c|}
\hline$\mathcal F$&$\varphi_{\subseteq[5]}(\mathcal F)$\\\hline
$(\textcolor{gray}{136789},\textcolor{gray}6,\textcolor{gray}1),(\textcolor{red}2,\textcolor{red}1,1),(\textcolor{blue}{45},\textcolor{blue}2,\textcolor{teal}1)$&$(\textcolor{gray}{136789},\textcolor{gray}6,\textcolor{gray}1),(\textcolor{red}2\textcolor{blue}{45},\textcolor{red}1\textcolor{blue}2,1)$\\\hline
$(\textcolor{red}1,\textcolor{red}1,1),(\textcolor{blue}{23},\textcolor{blue}2,\textcolor{teal}2),(\textcolor{gray}{456789},\textcolor{gray}{33},\textcolor{gray}2)$&$(\textcolor{red}1\textcolor{blue}{32},\textcolor{red}1\textcolor{blue}2,1),(\textcolor{gray}{456789},\textcolor{gray}{33},\textcolor{gray}2)$\\\hline
$(\textcolor{red}1,\textcolor{red}1,1),(\textcolor{gray}2\textcolor{orange}5\textcolor{gray}{6789},\textcolor{gray}{123},\textcolor{gray}1),(\textcolor{orange}{34},\textcolor{blue}2,\textcolor{teal}1)$&$(\textcolor{red}1\textcolor{orange}{53},\textcolor{red}1\textcolor{blue}2,1),(\textcolor{gray}2\textcolor{orange}4\textcolor{gray}{6789},\textcolor{gray}{123},\textcolor{gray}1)$\\\hline
$(\textcolor{red}1,\textcolor{red}1,1),(\textcolor{gray}{26}\textcolor{orange}4\textcolor{gray}{789},\textcolor{gray}6,\textcolor{gray}4),(\textcolor{orange}{35},\textcolor{blue}2,\textcolor{teal}2)$&$(\textcolor{red}1\textcolor{orange}{45},\textcolor{red}1\textcolor{blue}2,1),(\textcolor{gray}{26}\textcolor{orange}3\textcolor{gray}{789},\textcolor{gray}6,\textcolor{gray}4)$\\\hline
$(\textcolor{red}1,\textcolor{red}1,1),(\textcolor{gray}2\textcolor{orange}4\textcolor{gray}6\textcolor{orange}3\textcolor{gray}{789},\textcolor{gray}7,\textcolor{gray}5),(\textcolor{orange}5,\textcolor{blue}1,\textcolor{teal}1)$&$(\textcolor{red}1\textcolor{orange}4,\textcolor{red}1\textcolor{blue}1,1),(\textcolor{gray}2\textcolor{orange}3\textcolor{gray}6\textcolor{orange}5\textcolor{gray}{789},\textcolor{gray}7,\textcolor{gray}5)$\\\hline
\end{tabular}
\end{equation*}
\end{figure}
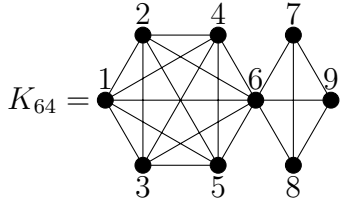

It remains to find a bijection $\varphi_{\subseteq[a-1]}$ between the case where $i'=i+1$ and $\mathcal T_{i+1}$ is not breakable, and the case where $i'=i$ and $\mathcal T_{i-1}$ is breakable with $\text{easybreak}(\mathcal T_{i-1})=(\mathcal S_1,\mathcal S_2=(S_2,\beta^{(2)},r'))$, but this time $\min(V(T_i))<\min(V(S_2))$. Then we get a sign-reversing involution by either applying $\varphi_{\subseteq[a-1]}$ or its inverse. Suppose that $\mathcal F$ is in the former case.\\

We would like to join $\mathcal T_{i+1}$ to $\mathcal T_{i-1}$, but the problem is that when we compute weight, the entries of $T_{i+1}$ would now be read before the entries of $T_i$. The idea is to first move $\mathcal T_{i+1}$ ``through'' $\mathcal T_i$, preserving the relative order of entries $x$ with $\min(V(T_i))<x<a$, before joining to $\mathcal T_{i-1}$. Let $k=|V(T_i)\cap [a]|$, $\sigma=\text{list}(T_i)$, $\sigma'=\text{startr}(\text{list}(T_{i+1}),r_{i+1})$, and let $2\leq j\leq k$ be such that $\sigma_j=a$. We now define
\begin{equation}
\tau=\sigma_2\ \cdots \ \sigma_{j-1} \ \sigma_{j+1} \ \cdots \ \sigma_k \ \sigma'_1 \ \cdots \ \sigma'_{\ell(\sigma')},
\end{equation}
the permutation consisting of the entries $x$ in $V(T_i)\sqcup V(T_{i+1})$ with $\min(V(T_i))<x<a$, which in particular are all adjacent. We use the first $|V(T_{i+1})|$ entries of $\tau$ to join with $\mathcal T_{i-1}$ to make the tree triple
\begin{equation}
\mathcal T=(\text{list}(T_{i-1})\cdot(\tau_1 \ \cdots \ \tau_{|V(T_{i+1})|}),\alpha^{(i-1)}\cdot\alpha^{(i+1)},r_{i-1}),
\end{equation}
and we use the remaining entries of $\tau$ to modify $\mathcal T_i$ to make the tree triple
\begin{equation}
\mathcal T_i'=(\sigma_1 \ \tau_{|V(T_{i+1})|+1} \ \cdots \ \tau_{|V(T_{i+1})|+j-2} \ a \ \tau_{|V(T_{i+1})|+j-1} \ \cdots \ \tau_{\ell(\tau)} \ \sigma_{k+1} \ \cdots \ \sigma_{\ell(\sigma)},\alpha^{(i)},r_i).
\end{equation}
Then we define $\varphi_{\subseteq[a-1]}(\mathcal F)$ by replacing by tree triples $\mathcal T_{i-1}$, $\mathcal T_i$, and $\mathcal T_{i+1}$ by $\mathcal T$ and $\mathcal T_i'$. The entries that moved are in $[a-1]$ so are adjacent in $G$ and these are indeed tree triples. Note that entries of $\tau$ occurring after $a$, which make a $G$-inversion with $a$ and not $\sigma_1$, are exchanged with entries occurring before $\sigma_1$, which make a $G$-inversion with $\sigma_1$ and not $a$. Therefore, the map $\varphi_{\subseteq[a-1]}$ preserves weight, and by construction it also reverses sign and preserves type and $\alpha^{(1)}_1$. We have indeed produced a forest triple for which $i'=i$, $\mathcal T_{i-1}$ is breakable with $\text{easybreak}(\mathcal T_{i-1})=(\mathcal S_1,\mathcal S_2=(S_2,\beta^{(2)},r'))$, and $\min(V(T_i))<\min(V(S_2))$.\\

To define the inverse map, we break $\mathcal T_{i-1}$ as $\text{easybreak}(\mathcal T_{i-1})=(\mathcal S_1,\mathcal S_2=(S_2,\beta^{(2)},r'))$ and move $\mathcal S_2$ back ``through'' $\mathcal T_i$. Let $k=|V(T_i)\cap[a]|$, $\sigma=\text{list}(T_i)$,  $\sigma'=\text{startr}(\text{list}(S_2),r')$, and let $2\leq j\leq k$ be such that $\sigma_j=a$. We now define
\begin{equation}\tau=\sigma'_1\ \cdots \ \sigma'_{\ell(\sigma')} \ \sigma_2 \ \cdots \ \sigma_{j-1} \ \sigma_{j+1} \ \cdots \ \sigma_k,\end{equation}
the permutation consisting of the entries $x$ in $V(S_2)\sqcup V(T_i)$ with $\min(V(T_i))<x<a$. We use the first $(k-2)$ entries of $\tau$ to modify $\mathcal T_i$ to make the tree triple
\begin{equation}
\mathcal T_i'=(\sigma_1 \ \tau_1 \ \cdots \ \tau_{j-2} \ a \ \tau_{j-1} \ \cdots \ \tau_{k-2} \ \sigma_{k+1} \ \cdots \ \sigma_{\ell(\sigma)}, \alpha^{(i)}, r_i),
\end{equation}
and we use the remaining entries of $\tau$ to make the tree triple
\begin{equation}
\mathcal S_2'=(\text{startmin}(\tau_{k-1} \ \cdots \ \tau_{\ell(\tau)}),\beta^{(2)},\text{indstart}(\tau_{k-1}\ \cdots \ \tau_{\ell(\tau)})).
\end{equation}
Then the inverse map is given by replacing the tree triples $\mathcal T_{i-1}$ and $\mathcal T_i$ by $\mathcal S_1$, $\mathcal T_i'$, and $\mathcal S_2'$. \\

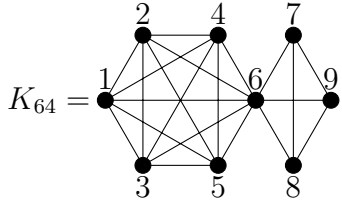
\begin{figure}
\caption{\label{fig:varphi3} Some simple forest triples $\mathcal F$ of $K_{64}$ and $\varphi_{\geq 3}(\mathcal F)$}

\begin{tikzpicture}
\draw (1.25,0) node (){$K_{64}=$};
\filldraw (2,0) circle (3pt) node[align=center,above] (1){1};
\filldraw (2.5,0.866) circle (3pt) node[align=center,above] (2){2};
\filldraw (2.5,-0.866) circle (3pt) node[align=center,below] (3){3};
\filldraw (3.5,0.866) circle (3pt) node[align=center,above] (4){4};
\filldraw (3.5,-0.866) circle (3pt) node[align=center,below] (5){5};
\filldraw (4,0) circle (3pt) node[align=center,above] (6){6};
\filldraw (4.5,0.866) circle (3pt) node[align=center,above] (7){7};
\filldraw (4.5,-0.866) circle (3pt) node[align=center,below] (8){8};
\filldraw (5,0) circle (3pt) node[align=center,above] (9){9};

\draw (2,0)--(5,0) (2,0) -- (2.5,0.866) -- (3.5,0.866) -- (4,0) -- (3.5,-0.866) -- (2.5,-0.866) -- (2,0) (2,0) -- (3.5,-0.866) -- (3.5,0.866) -- (2,0) (4,0) -- (2.5,0.866) -- (2.5,-0.866) -- (4,0) (2.5,0.866) -- (3.5,-0.866) (2.5,-0.866) -- (3.5,0.866) (4,0) -- (4.5,0.866) -- (5,0) -- (4.5,-0.866) -- (4,0) (4.5,0.866) -- (4.5,-0.866);
\end{tikzpicture}
\begin{equation*}
\begin{tabular}{|c|c|}\hline
$\mathcal F$&$\varphi_{\geq 3}(\mathcal F)$\\\hline
$(1\textcolor{red}2\textcolor{teal}6\textcolor{blue}{345}\textcolor{gray}{789},\textcolor{red}1\textcolor{blue}3\textcolor{red}5,1)$&$(\textcolor{red}{126}\textcolor{gray}{789},\textcolor{red}{15},1),(\textcolor{blue}{345},\textcolor{blue}3,\textcolor{blue}1)$\\\hline
$(1\textcolor{teal}6\textcolor{red}3\textcolor{blue}{425}\textcolor{gray}{789},\textcolor{red}1\textcolor{blue}3\textcolor{red}5,1)$&$(\textcolor{red}{163}\textcolor{gray}{789},\textcolor{red}{15},1),(\textcolor{blue}{245},\textcolor{blue}3,\textcolor{blue}2)$\\\hline
$(1\textcolor{red}5\textcolor{teal}6\textcolor{blue}{432}\textcolor{gray}{789},\textcolor{red}1\textcolor{blue}3\textcolor{red}{41},1)$&$(\textcolor{red}{156}\textcolor{gray}{789},\textcolor{red}{141},1),(\textcolor{blue}{243},\textcolor{blue}3,\textcolor{blue}3)$\\\hline
$(1\textcolor{orange}2345\textcolor{teal}6\textcolor{gray}{789},\textcolor{red}1\textcolor{blue}{35},1)$&$(\textcolor{red}1,\textcolor{red}1,1),(\textcolor{blue}{2345}\textcolor{orange}6\textcolor{gray}{789},\textcolor{blue}{35},\textcolor{teal}1)$\\\hline
$(1\textcolor{orange}3245\textcolor{teal}6\textcolor{gray}{789},\textcolor{red}1\textcolor{blue}{35},1)$&$(\textcolor{red}1,\textcolor{red}1,1),(\textcolor{blue}{234}\textcolor{orange}6\textcolor{blue}5\textcolor{gray}{789},\textcolor{blue}{35},\textcolor{teal}1)$\\\hline
$(1\textcolor{orange}234\textcolor{teal}65\textcolor{gray}{789},\textcolor{red}1\textcolor{blue}{35},1)$&$(\textcolor{red}1,\textcolor{red}1,1),(\textcolor{blue}{2345}\textcolor{orange}6\textcolor{gray}{789},\textcolor{blue}{35},\textcolor{teal}2)$\\\hline
$(1\textcolor{orange}45\textcolor{teal}623\textcolor{gray}{789},\textcolor{red}1\textcolor{blue}{341},1)$&$(\textcolor{red}1,\textcolor{red}1,1),(\textcolor{blue}{25}\textcolor{orange}6\textcolor{blue}{34}\textcolor{gray}{789},\textcolor{blue}{341},\textcolor{teal}3)$\\\hline
\end{tabular}
\end{equation*}
\end{figure}

\textbf{Case 4: $\mathcal F\vert_{\geq a}$ is fixed under $\varphi'$, $\text{seg}_{[a]}(\mathcal F)\geq 3$, and $\text{seg}_{[a]}(\mathcal T_i)-\text{seg}_{[a]}(\mathcal F)=0$ or $1$.} We define a bijection $\varphi_{\geq 3}$ between such $\mathcal F$ with $\text{seg}_{[a]}(\mathcal T_i)=\text{seg}_{[a]}(\mathcal F)$, meaning that the only tree triple present in $[a]$ is $\mathcal T_1=(T_1,\alpha^{(1)},1)$, where $[a]\subseteq V(T_1)$, $\ell(\alpha^{(1)})\geq 3$, and $\alpha^{(1)}_1+\alpha^{(1)}_2<a$; and those with $\text{seg}_{[a]}(\mathcal T_i)=\text{seg}_{[a]}(\mathcal F)-1$, meaning there is exactly one other tree triple present in $[a]$ and it is not breakable. Then we get a sign-reversing involution by either applying $\varphi_{\geq 3}$ or its inverse. Let $\sigma=\text{list}(T_1)$ and let $j$ be such that $\sigma_j=a$. Note that $2\leq j\leq a$ by Lemma \ref{lem:treelistofkchain} and we can have $j=2$ only if $\epsilon=0$. We will have two cases, depending on the position of the entry $a$ in $\sigma$.\\

If $j\leq a-\alpha^{(1)}_2$, then we remove $\alpha^{(1)}_2$ entries from the inside of $\sigma$ and define
\begin{equation}  \sigma^{\text{in}}=\sigma_{a-\alpha^{(1)}_2+1}\cdots\sigma_a\text{ and }\sigma^{\text{out}}=\sigma_1\cdots \sigma_{a-\alpha^{(1)}_2} \ \sigma_{a+1} \ \cdots \ \sigma_{|V(T)|},\end{equation} and we define $\varphi_{\geq 3}(\mathcal F)$ by replacing $\mathcal T_1$ by the tree triples
\begin{equation}
\mathcal S_1=(\sigma^{\text{out}},\alpha^{(1)}_1\alpha^{(1)}_3\cdots\alpha^{(1)}_\ell,1)\text{ and }\mathcal S_2=(\text{startmin}(\sigma^{\text{in}}),\alpha^{(1)}_2,\text{indstart}(\sigma^{\text{in}})).
\end{equation}
Because $j\leq a-\alpha^{(1)}_2$, the entry $a$ appears in $\sigma^{\text{out}}$ so its descents and successive LR maxima appear in $\sigma$, and $\text{startmin}(\sigma^{\text{in}})$ has all entries in $[a-1]$ so these are tree lists of $G$, and we have $r_2\leq\ell(\sigma^{\text{in}})=\alpha^{(1)}_2$, so $\mathcal S_2$ is indeed a tree triple of $G$. Moving the string $\sigma^{\text{in}}$ does not affect the number of $G$-inversions because its entries are less than $a$, whereas every $\sigma_t>a$ for $t>a$ by Lemma \ref{lem:treelistofkchain}. The map $\varphi_{\geq 3}$ reverses sign, preserves type, $\alpha^{(1)}_1$, and weight. Also note that the vertices $1$ and $a$ are in the same tree.\\

If $a-\alpha^{(1)}_2+1\leq j\leq a$, which in particular means that $j\neq 2$ because $\alpha^{(1)}_1+\alpha^{(1)}_2<a$, then we separate $\sigma$ into its head and tail,
\begin{equation}
\sigma^{\text{head}}=\sigma_1\cdots\sigma_{\alpha^{(1)}_1}\text{ and }\sigma^{\text{tail}}=\sigma_{\alpha^{(1)}_1+1}\cdots\sigma_{|V(T)|}. 
\end{equation}
As before, we apply transpositions to $\sigma^{\text{tail}}$ to put its minimum element in front. We will now do something a little unusual. We move the entry $a$ to position $(a-\text{indstart}(\sigma^{\text{tail}})+1)$ so that it makes $(\text{indstart}(\sigma^{\text{tail}})-1)$ $G$-inversions to account for these transpositions. We then use the new value of $r=a-j+1$ to account for the original $(a-j)$ $G$-inversions made by the entry $a$. Specifically, let $r'=\text{indstart}(\sigma^{\text{tail}})$ and let $\tau$ be $\text{startmin}(\sigma^{\text{tail}})$ with the entry $a$ removed, and define $\varphi_{\geq 3}(\mathcal F)$ by replacing $\mathcal T_1$ by the tree triples
\begin{equation}
\mathcal S_1=(\sigma^{\text{head}},\alpha^{(1)}_1,1)\text{ and }\mathcal S_2=(\tau_1 \ \cdots \ \tau_{a-r'} \ a \ \tau_{a-r'+1} \ \cdots \ \tau_{\ell(\tau)},\alpha^{(1)}_2\cdots\alpha^{(1)}_\ell,a-j+1).
\end{equation}
The entries of $\sigma^{\text{head}}$ and the entries of $\tau$ that moved are all in $[a-1]$ so are adjacent and these are tree lists of $G$. We also have $a-j+1\leq\alpha^{(2)}_1$, so $\mathcal S_2$ is indeed a tree triple of $G$. The map $\varphi_{\geq 3}$ reverses sign and preserves type, $\alpha^{(1)}_1$, and weight. Also note that the vertices $1$ and $a$ are in different trees. \\

It remains to show that the map $\varphi_{\geq 3}$ is a bijection. Suppose that $\text{seg}_{[a]}(\mathcal T_i)=\text{seg}_{[a]}(\mathcal F)-1$. The idea is that we can determine which case of the map was applied by whether the vertices $1$ and $a$ are in the same tree in $\mathcal F$, so we can recover the original tree triple.\\

If the vertices $1$ and $a$ are in the same tree $T_i=T_1$, then we must have been in the first case where $j\leq a-\alpha^{(1)}_2$. Let $\sigma=\text{list}(T_1)$, $\sigma'=\text{startr}(\text{list}(T_2),r_2)$, $k=|V(T_2)\cap [a]|$, and then the inverse map is given by inserting $\sigma'$ back into $\sigma$, specifically, by replacing $\mathcal T_1$ and $\mathcal T_2$ by 
\begin{equation}
\mathcal T=(\sigma_1 \ \cdots \ \sigma_k \ \sigma'_1\ \cdots \ \sigma'_{|V(T_2)|} \ \sigma_{k+1} \ \cdots \ \sigma_{|V(T_1)|}, \alpha^{(1)}_1\ \alpha^{(2)}_1 \ \alpha^{(1)}_2 \ \cdots \ \alpha^{(1)}_\ell,1).
\end{equation}
If the vertices $1$ and $a$ are in different trees, so that $1\in V(T_1)$ and $a\in V(T_2)$, then we must have been in the second case where $a-\alpha^{(1)}_2+1\leq j\leq a$. Let $\sigma'=\text{list}(T_1)$, let $\sigma=\text{list}(T_2)$, let $k=|V(T_2)\cap [a]|$, let $j'$ be such that $\sigma_{j'}=a$, and let $\tau$ be $\text{startr}(\text{list}(T_2),k-j'+1)$ with the entry $a$ removed. The position of the entry $a$ tells us the original first entry of the tail and the value of $r_2$ tells us the original position of the entry $a$. The inverse map is given by replacing $\mathcal T_1$ and $\mathcal T_2$ by
\begin{equation}
\mathcal T=(\sigma'_1 \ \cdots \ \sigma'_{|V(T_2)|} \ \tau_1 \ \cdots \ \tau_{k-r_2-1} \ a \ \tau_{k-r_2} \ \cdots \ \tau_{\ell(\tau)},\alpha^{(1)}_1 \ \alpha^{(2)}_1 \ \cdots \ \alpha^{(2)}_\ell,1).
\end{equation}
These constructions are inverse to each other, so we indeed have a sign-reversing involution in this case.\\

By combining our sign-reversing involutions in these four cases, we build a nice involution $\varphi$ for $G$ with the prescribed fixed points, thus completing the proof.
\end{proof}

Now we can repeatedly apply Theorem \ref{thm:key} to obtain a combinatorial description of fixed forest triples for any \emph{almost-$K$-chain}, which is a graph of the form \begin{equation}K_\gamma^\epsilon=K_{\gamma_1}^{\epsilon_1}+\cdots+K_{\gamma_\ell}^{\epsilon_\ell},\end{equation} where $\gamma=\gamma_1\cdots\gamma_\ell$ is a composition with all parts at least $2$ and $\epsilon=\epsilon_1\cdots\epsilon_\ell$ is a sequence of $0$'s and $1$'s. Aliniaeifard, Wang, and van Willigenburg proved that for such graphs, the chromatic symmetric function $X_{K_\gamma^\epsilon}(\bm x)$ is $e$-positive \cite[Proposition 5.4]{csfvertex}.

\begin{definition}\label{def:kchainfixed}
Let $K_\gamma^\epsilon$ be an almost-$K$-chain and for $0\leq t\leq\ell(\gamma)$, let \begin{equation}c_t=\gamma_1+\cdots+\gamma_t-t+1,\end{equation} so that $c_0=1$, $c_{\ell(\gamma)}=n$, and $c_1,\ldots,c_{\ell(\gamma)-1}$ are the cut vertices of $K_\gamma^\epsilon$. Let $\text{Fix}(K_\gamma^\epsilon)$ denote the set of forest triples $\mathcal F\in\text{FT}(K_\gamma^\epsilon)$ that are atoms and such that for every $1\leq t\leq\ell(\gamma)$, letting $\mathcal T=(T,\alpha,r)$ denote the tree triple of $\mathcal F$ with $c_t\in V(T)$, we have the following.
\begin{enumerate}
\item [(C1)] There is at most one other tree triple $\mathcal T'=(T',\alpha',r')$ in $\mathcal F$ such that there exists a vertex $v\in V(T')$ with $c_{t-1}\leq v\leq c_t$.
\item [(C2)] If $c_{t-1}\notin V(T)$, we have $r\geq \gamma_t-\epsilon_t$.
\item [(C3)] If $c_{t-1}\in V(T)$, we have $|\{v\in V(T): \ v\geq c_{t-1}\}|\geq \gamma_t$. 
\end{enumerate}
Let $\text{Fix}_{\text{simple}}(K_\gamma^\epsilon)$ denote the set of simple forest triples $\mathcal F\in\text{Fix}(K_\gamma^\epsilon)$.
\end{definition}

\begin{corollary} The chromatic quasisymmetric function of an almost-$K$-chain $K_\gamma^\epsilon$ is
\begin{equation} \label{eq:kchainfixed}
X_{K_\gamma^\epsilon}(\bm x;q)=\sum_{\mathcal F\in\text{Fix}_{\text{simple}}(K_\gamma^\epsilon)}[\alpha^{(1)}_1(\mathcal F)]_q \ q^{\text{weight}(\mathcal F)}e_{\text{type}(\mathcal F)}=\sum_{\mathcal F\in\text{Fix}(K_\gamma^\epsilon)}q^{\text{weight}(\mathcal F)}e_{\text{type}(\mathcal F)}.
\end{equation} 
In particular, $X_{K_\gamma^\epsilon}(\bm x;q)$ is $e$-positive.
\end{corollary}

\begin{proof}
We use induction on $\ell(\gamma)$ to show that there is a nice involution whose fixed points are exactly $\text{Fix}_{\text{simple}}(K_\gamma^\epsilon)$. The case where $\ell(\gamma)=1$ is done in Proposition \ref{prop:complete}, so suppose that $\ell(\gamma)\geq 2$ and let $G'=K_{\gamma_2}^{\epsilon_2}+\cdots+K_{\gamma_\ell}^{\epsilon_\ell}$. By Theorem \ref{thm:key} and our induction hypothesis, there is a nice involution $\varphi$ for $K_\gamma^\epsilon=K_{\gamma_1}^{\epsilon_1}+G'$ with fixed points
exactly
\begin{align}\label{eq:kchainfixedinduction}
\text{Fix}(\varphi)=\{\mathcal F\in & \text{FT}_{\text{simple}}(K_\gamma^\epsilon): \ \mathcal F\text{ is an atom, }\mathcal F\vert_{\geq \gamma_1}\in\text{Fix}_{\text{simple}}(G'), \ \text{seg}_{[\gamma_1]}(\mathcal F)\leq 2,\\\nonumber&\text{ if }1\notin V(T_i),\text{ then }r_i\geq \gamma_1-\epsilon_1,\text{ if }1\in V(T_i),\text{ then }|V(T_i)|\geq \gamma_1\},
\end{align} where $\mathcal T_i=(T_i,\alpha^{(i)},r_i)$ is the tree triple of $\mathcal F$ with $\gamma_1\in V(T_i)$. Note that by induction, the condition $\mathcal F\vert_{\geq \gamma_1}\in\text{Fix}_{\text{simple}}(G')$ means that for the restricted tree triple $\mathcal T_i\vert_{\geq \gamma_1}=(T',\alpha',1)$, if $\gamma_2\in V(T')$, then we must have $|V(T')|\geq\gamma_2$, which means that for the tree triple $\mathcal T_i$, if $c_1\in V(T_i)$, we must have 
$|\{v\in V(T_i): \ v\geq c_1\}|\geq \gamma_2$, because it is only these vertices of $T_i$ that appear in the restriction to $G'$. The other conditions of $\text{Fix}(K_\gamma^\epsilon)$ are exactly those in \eqref{eq:kchainfixedinduction}, so we have $\text{Fix}(\varphi)=\text{Fix}_{\text{simple}}(K_\gamma^\epsilon)$, as desired.
\end{proof}

Now by enumerating these fixed points, we can simplify \eqref{eq:kchainfixed} to a more explicit expression. We first state a Lemma about $q$-counting subsets.

\begin{lemma}\label{lem:qbinom}
For integers $0\leq k\leq n$, we have
\begin{equation}
\sum_{S\subseteq [n]: \ |S|=k}q^{|\{(i,j): \ i>j, \ i\notin S, \ j\in S\}|}=\binom nk_q,
\end{equation}
where the \emph{$q$-binomial coefficient} is $\binom nk_q=\frac{[n]_q!}{[k]_q![n-k]_q!}$.
\end{lemma}

\begin{proof}
As with ordinary binomial coefficients, the result follows from induction on $n$, the identity $\binom nk_q=q^k\binom{n-1}k_q+\binom{n-1}{k-1}_q$, and from considering the cases of $n\notin S$ and $n\in S$.
\end{proof}

Before we prove our explicit formula for $X_{K_\gamma^\epsilon}(\bm x;q)$, we will first consider the simpler case of a $K$-chain with two cliques because this will demonstrate the main ideas. 

\begin{corollary}
Consider the $K$-chain $K_{ab}$ with two cliques of sizes $a$ and $b$, which has $n=a+b-1$ vertices. Then the chromatic quasisymmetric function of $K_{ab}$ is
\begin{equation}
X_{K_{ab}}(\bm x;q)=[a-1]_q![b-1]_q!\sum_{k=\max\{a,b\}}^nq^{n-k}[2k-n]_qe_{k(n-k)}.
\end{equation}
\end{corollary}

\begin{proof}
We count fixed forest triples $\mathcal F\in\text{Fix}(K_{ab})$ by weight. Let $\mathcal T=(T,\alpha,r)$ be the tree triple of $\mathcal F$ with $n\in V(T)$ and let $k=|V(T)|$. We must have $a\in V(T)$, otherwise $V(T)\subseteq\{a+1,\ldots,n\}$ and $r\leq k\leq b-1$, while (C2) requires $r\geq b$. Now by (C3), we must have $\{a,\ldots,n\}\subseteq V(T)$ and $k\geq b$. If $1\notin V(T)$, (C2) implies that $k\geq r\geq a$, while if $1\in V(T)$, (C3) again implies that $k\geq a$, so we must have $k\geq\max\{a,b\}$. By (C1), there is at most one other tree triple $\mathcal T'=(T',\alpha',r')$ in $\mathcal F$, and we have $|V(T')|=n-k$ and $V(T')\subseteq [a-1]$. We now consider the possibilities of $\mathcal F$ with $1\in V(T')$ and with $1\in V(T)$.\\

If $1\in V(T')$, then we must have $r\geq a$ by (C2). By using Lemma \ref{lem:sumovertreelists} to account for $G$-inversions within the trees $T$ and $T'$ and using Lemma \ref{lem:qbinom} to account for $G$-inversions between trees $T'$ and $T$, the $q$-weighted sum over these fixed points is 

\begin{align}\label{eq:nonnestingsum2}
\sum_{\substack{\mathcal F\in\text{Fix}(K_{ab}) \\ |V(T)|=k \\1\in V(T')}}q^{\text{weight}(\mathcal F)}&=
\binom{a-2}{k-b}_q[n-k-1]_q![k-b]_q![b-1]_q![n-k]_q([k]_q-[a-1]_q)\\\nonumber&=q^{a-1}[a-2]_q![b-1]_q![n-k]_q[k-a+1]_q.
\end{align}
If $1\in V(T)$, then we can have any $1\leq r\leq k$. We have $(n-k)$ $G$-inversions between vertex $n\in V(T)$ and the vertices in $T'$. By Lemma \ref{lem:sumovertreelists} and Lemma \ref{lem:qbinom}, we have
\begin{align}\label{eq:nestingsum2}
\sum_{\substack{\mathcal F\in\text{Fix}(K_{ab}) \\ |V(T)|=k\\ 1\in V(T)}}q^{\text{weight}(\mathcal F)}&=
q^{n-k}\binom{a-2}{n-k}_q[n-k-1]_q![k-b]_q![b-1]_q![n-k]_q[k]_q\\\nonumber&=q^{n-k}[a-2]_q![b-1]_q![k-b]_q[k]_q.
\end{align}
Now by summing \eqref{eq:nonnestingsum2} and \eqref{eq:nestingsum2} over all $\max\{a,b\}\leq k\leq n$, we have as desired
\begin{align*}
X_{K_{ab}}(\bm x;q)&=[a-2]_q![b-1]_q!\sum_{k=\max\{a,b\}}^n(q^{a-1}[n-k]_q[k-a+1]_q+q^{n-k}[k-b]_q[k]_q)e_{k(n-k)}\\&=[a-1]_q![b-1]_q!\sum_{k=\max\{a,b\}}^nq^{n-k}[2k-n]_qe_{k(n-k)}.
\end{align*}
\end{proof}

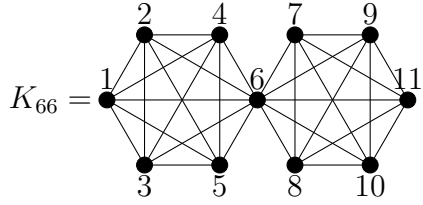
\begin{figure}
\caption{\label{fig:kchain66example} The $K$-chain $K_{66}$ and the chromatic quasisymmetric function $X_{K_{66}}(\bm x;q)$}
\begin{tikzpicture}
\draw (1.25,0) node (){$K_{66}=$};
\filldraw (2,0) circle (3pt) node[align=center,above] (5){1};
\filldraw (2.5,0.866) circle (3pt) node[align=center,above] (6){2};
\filldraw (3.5,0.866) circle (3pt) node[align=center,above] (7){4};
\filldraw (2.5,-0.866) circle (3pt) node[align=center,below] (8){3};
\filldraw (3.5,-0.866) circle (3pt) node[align=center,below] (9){5};
\filldraw (4,0) circle (3pt) node[align=center,above] (10){6};

\draw (2,0)--(6,0);
\draw (2,0) -- (2.5,0.866) -- (3.5,0.866) -- (4,0) -- (3.5,-0.866) -- (2.5,-0.866) -- (2,0);
\draw (2,0) -- (3.5,-0.866) -- (3.5,0.866) -- (2,0);
\draw (4,0) -- (2.5,0.866) -- (2.5,-0.866) -- (4,0);
\draw (2.5,0.866) -- (3.5,-0.866);
\draw (2.5,-0.866) -- (3.5,0.866);
\filldraw (4.5,0.866) circle (3pt) node[align=center,above] (6){7};
\filldraw (5.5,0.866) circle (3pt) node[align=center,above] (7){9};
\filldraw (4.5,-0.866) circle (3pt) node[align=center,below] (8){8};
\filldraw (5.5,-0.866) circle (3pt) node[align=center,below] (9){10};
\filldraw (6,0) circle (3pt) node[align=center,above] (10){11};

\draw (4,0) -- (4.5,0.866) -- (5.5,0.866) -- (6,0) -- (5.5,-0.866) -- (4.5,-0.866) -- (4,0);
\draw (4,0) -- (5.5,-0.866) -- (5.5,0.866) -- (4,0);
\draw (6,0) -- (4.5,0.866) -- (4.5,-0.866) -- (6,0);
\draw (4.5,0.866) -- (5.5,-0.866);
\draw (4.5,-0.866) -- (5.5,0.866);
\end{tikzpicture}
\begin{equation}
X_{K_{66}}(\bm x;q)=[5]_q![5]_q!(q^5e_{65}+q^4[3]_qe_{74}+q^3[5]_qe_{83}+q^2[7]_qe_{92}+q[9]_qe_{(10)1}+[11]_qe_{(11)}).
\end{equation}
\end{figure}

We now state the general formula for any almost-$K$-chain. 

\begin{definition}
Let $\gamma=\gamma_1\cdots\gamma_\ell$ be a composition with all parts at least $2$ and let $\epsilon=\epsilon_1\cdots\epsilon_\ell$ be a list of $0$'s and $1$'s. We define $A_\gamma^\epsilon$ to be the set of weak compositions $\alpha=\alpha_1\cdots\alpha_{\ell+1}$ of length $\ell+1$ and size $|\gamma|-\ell+1$ such that $\alpha_1\geq 1$ and for each $2\leq i\leq \ell+1$ we have either
\begin{align}\label{eq:Acondition1}
\alpha_i<\gamma_{i-1}-\epsilon_{i-1}-1&\text{ and }\alpha_i+\cdots+\alpha_{\ell+1}<\gamma_i+\cdots+\gamma_\ell-(\ell-i),\text{ or }\\\label{eq:Acondition2}
\alpha_i\geq\gamma_{i-1}-\epsilon_{i-1}&\text{ and }\alpha_i+\cdots+\alpha_{\ell+1}\geq\gamma_i+\cdots+\gamma_\ell-(\ell-i).
\end{align}
\end{definition}

\begin{remark}\label{rem:alphaell=0}
Note that for $i=\ell+1$, the conditions are that $\alpha_{\ell+1}=0$ or $\alpha_{\ell+1}\geq\gamma_\ell-\epsilon_\ell$. In particular, if $\epsilon_\ell=0$, we must have $\alpha_{\ell+1}=0$, otherwise these conditions will force $\alpha_{\ell+1}\geq\gamma_\ell$ and $\alpha_i\geq\gamma_{i-1}-1$ for every $2\leq i\leq\ell$, but then $|\alpha|\geq|\gamma|-\ell+2$.
\end{remark}

\begin{example}
If $\gamma=abc$ and $\epsilon=000$, then $A_{abc}^{000}$ is the set of weak compositions $\alpha=\alpha_1\alpha_2\alpha_3\alpha_4$ such that $\alpha_1+\alpha_2+\alpha_3+\alpha_4=a+b+c-2$, $\alpha_1\geq 1$, $\alpha_4=0$ by Remark \ref{rem:alphaell=0}, and
\begin{align}
\alpha_2<a-1\text{ and }\alpha_2+\alpha_3<b+c-1,&\text{ or }\alpha_2\geq a\text{ and }\alpha_2+\alpha_3\geq b+c-1,\\
\alpha_3<b-1\text{ and }\alpha_3<c,&\text{ or }\alpha_3\geq b\text{ and }\alpha_3\geq c.
\end{align}
\end{example}

\begin{corollary}\label{cor:kchainexplicit}
The chromatic quasisymmetric function of an almost-$K$-chain $K_\gamma^\epsilon$ is 
\begin{equation}
\label{eq:kchainexplicit}
X_{K_\gamma^\epsilon}(\bm x;q)=[\gamma_1-2]_q!\cdots[\gamma_\ell-2]_q!\sum_{\alpha\in A_\gamma^\epsilon}[\alpha_1]_q\prod_{i=2}^{\ell+1}q^{m_i}[|\alpha_i-(\gamma_{i-1}-1-\epsilon_{i-1})|]_qe_{\text{sort}(\alpha)},
\end{equation}
where $\ell=\ell(\gamma)$ and $m_i=\min\{\alpha_i,\gamma_{i-1}-1-\epsilon_{i-1}\}$. 
\end{corollary}

\begin{proof}
The idea is to associate a weak composition $\alpha\in A_\gamma^\epsilon$ to each fixed point $\mathcal F\in\text{Fix}(K_\gamma^\epsilon)$ and then to enumerate $q^{\text{weight}(\mathcal F)}$ over all $\mathcal F$ corresponding to a particular $\alpha$. Given a fixed forest triple $\mathcal F$, we define $\alpha=\alpha_1\cdots\alpha_{\ell+1}$ as follows. For $1\leq i\leq\ell+1$, let $T_i$ be the tree with $c_{i-1}\in V(T_i)$. Note that different indices $i$ may refer to the same tree because for example it is possible that $c_{i-2}\in V(T_i)$ as well. Let $\alpha_1=|V(T_1)|\geq 1$ be the size of the tree containing vertex $1$. Then for $2\leq i\leq\ell+1$, if $c_{i-2}\notin V(T_i)$, we let $\alpha_i=|V(T_i)|$ be the size of the tree $T_i$, while if $c_{i-2}\in V(T_i)$, then there is at most one other tree $T_i'$ with $C_i\cap V(T_i')\neq\emptyset$; we let $\alpha_i=|V(T_i')|$ be the size of this other tree $T_i'$ if there is one, otherwise we set $\alpha_i=0$. Note that all trees have been considered, so $|\alpha|=n=|\gamma|-\ell+1$. \\

If $c_{i-2}\notin V(T_i)$, meaning that $\alpha_i=|V(T_i)|$, then we must have $\alpha_i\geq \gamma_{i-1}-\epsilon_{i-1}$ by (C2) and the trees $T_j$ for $j\geq i$ must use all the vertices at least $c_{i-1}$, so we must have
\begin{equation}
\alpha_i+\cdots+\alpha_{\ell+1}\geq\gamma_i+\cdots+\gamma_\ell-(\ell-i).
\end{equation}
If $c_{i-2}\in V(T_i)$, meaning that $\alpha_i=|V(T_i')|$ is the size of the other tree $T_i'$ present in $C_{i-1}$, if any, then we must have $V(T_{i'})\subseteq\{c_{i-2}+1,\ldots,c_{i-1}-1\}$ so $\alpha_i<\gamma_{i-1}-1$, the tree $T_i$ must contain at least $\gamma_{i-2}$ vertices at least $c_{i-2}$ by (C3), and therefore the tree $T_i'$ along with the trees to the right of $T_i$ use the remaining vertices at least $c_{i-2}$, so we must have
\begin{equation}
\alpha_i+\cdots+\alpha_{\ell+1}<\gamma_i+\cdots+\gamma_{\ell}-(\ell-i).
\end{equation}
Also note that if $\epsilon_{i-1}=0$ and $\alpha_i=\gamma_{i-1}-\epsilon_{i-1}-1$, then the summand in \eqref{eq:kchainexplicit} is zero so we can equivalently replace the condition $\alpha_i<\gamma_{i-1}-1$ by the condition $\alpha_i<\gamma_{i-1}-\epsilon_{i-1}-1$, and therefore the weak compositions that arise are exactly those $\alpha\in A_\gamma^\epsilon$. We now enumerate $q^{\text{weight}(\mathcal F)}$ over all $\mathcal F$ associated to a particular $\alpha\in A_\gamma^\epsilon$. The factor $[\alpha_1]_q$ accounts for the possible choices of $1\leq r_1\leq\alpha_1$. Now for $2\leq i\leq\ell+1$, we consider the contributions given by the choices for $r_i$ and the inversion-weighted choices of how to join vertices in the clique $C_{i-1}$.\\

If $c_{i-2}\notin V(T_i)$, then we must have $r_i\geq\gamma_{i-1}-\epsilon_{i-1}$. Let $k=\alpha_1+\cdots+\alpha_{i-1}-c_{i-2}+1$ be the number of vertices of $T_i$ in $C_{i-1}$. Then by Lemma \ref{lem:sumovertreelists} and Lemma \ref{lem:qbinom}, the contribution given by the possible choices in $C_{i-1}$ is
\begin{align}
\label{eq:nonnestingsum} 
\binom{\gamma_{i-1}-2}{k-1}_q&[k-1]_q![\gamma_{i-1}-k-1]_q!([\alpha_i]_q-[\gamma_{i-1}-1-\epsilon_{i-1}]_q)\\\nonumber&=q^{\gamma_{i-1}-1-\epsilon_{i-1}}[\gamma_{i-1}-2]_q![\alpha_i-(\gamma_{i-1}-1-\epsilon_{i-1})]_q.
\end{align}

If $c_{i-2}\in V(T_i)$, then we can have any $1\leq r_i\leq \alpha_i$. We have $\alpha_i$ $G$-inversions between vertex $c_{i-1}\in V(T_i)$ and the vertices in $T_{i'}$. By Lemma \ref{lem:sumovertreelists} and Lemma \ref{lem:qbinom}, the contribution given by the possible choices in $C_{i-1}$ is
\begin{align}
\label{eq:nestingsum}
q^{\alpha_i}\binom{\gamma_{i-1}-2}{\alpha_i}_q&[\alpha_i-1]_q![\gamma_{i-1}-2-\alpha_i]_q![\gamma_{i-1}-1-\epsilon_{i-1}-\alpha_i]_q[\alpha_i]_q\\\nonumber&=q^{\alpha_i}[\gamma_{i-1}-2]_q![(\gamma_{i-1}-1-\epsilon_{i-1})-\alpha_i]_q.
\end{align}
Note that if $\alpha_i=0$, meaning there is no tree $T_i'$ present, the expression in \eqref{eq:nestingsum} is still valid. Therefore, in either case, we can use the absolute value to express the contribution as 
\begin{align}
q^{m_i}[\gamma_{i-1}-2]_q![|\alpha_i-(\gamma_{i-1}-1-\epsilon_{i-1})|]_q.
\end{align}
Multiplying these contributions over all $2\leq i\leq \ell+1$, we have as desired
\begin{align}
X_{K_\gamma^{\epsilon}}&(\bm x;q)=\sum_{\alpha\in A_\gamma^\epsilon}[\alpha_1]_q\prod_{i=2}^{\ell+1}q^{m_i}[\gamma_{i-1}-2]_q![|\alpha_i-(\gamma_{i-1}-1-\epsilon_{i-1})|]_qe_{\text{sort}(\alpha)}\\\nonumber&=[\gamma_1-2]_q!\cdots[\gamma_\ell-2]_q!\sum_{\alpha\in A_\gamma^\epsilon}[\alpha_1]_q\prod_{i=2}^{\ell+1}q^{m_i}[|\alpha_i-(\gamma_{i-1}-1-\epsilon_{i-1})|]_qe_{\text{sort}(\alpha)}.
\end{align}

\end{proof}

Our explicit formula also allows us to confirm $e$-unimodality for almost-$K$-chains. We make the following elementary observation.

\begin{observation}
Let $a(q)=\sum_{k=i}^ja_kq^k\in\mathbb N[q]$ and $b(q)=\sum_{k=i'}^{j'}b_kq^k\in\mathbb N[q]$ be positive palindromic unimodal polynomials with centers of symmetry $\frac{i+j}2$ and $\frac{i'+j'}2$. 
\begin{enumerate}
\item [(O1)] The product $a(q)b(q)$ is a positive palindromic unimodal polynomial with center of symmetry $\frac{i+j}2+\frac{i'+j'}2$.
\item [(O2)] If $\frac{i+j}2=\frac{i'+j'}2$, then the sum $a(q)+b(q)$ is a positive palindromic unimodal polynomial with the same center of symmetry.
\end{enumerate}
\end{observation}

\begin{corollary}\label{cor:kchainuni}
The chromatic quasisymmetric function of an almost-$K$-chain $K_\gamma^\epsilon$ is $e$-positive and $e$-unimodal.
\end{corollary}

\begin{proof}
By \eqref{eq:kchainexplicit}, the coefficient of $e_\mu$ in $X_{K_\gamma^\epsilon}(\bm x;q)$ is 
\begin{equation}\label{eq:kchaincmuq}
c_\mu(q)=[\gamma_1-2]_q!\cdots[\gamma_\ell-2]_q!\sum_{\alpha\in A_\gamma^\epsilon: \ \text{sort}(\alpha)=\mu}[\alpha_1]_q\prod_{i=2}^{\ell+1}q^{m_i}[|\alpha_i-(\gamma_{i-1}-1-\epsilon_{i-1})|]_q.
\end{equation}
Note that regardless of whether $\alpha_i<\gamma_{i-1}-1-\epsilon_{i-1}$ or $\alpha_i\geq\gamma_{i-1}-1-\epsilon_{i-1}$, the factor $q^{m_i}[|\alpha_i-(\gamma_{i-1}-1-\epsilon_{i-1})|]_q$ is a positive palindromic unimodal polynomial with center of symmetry $\frac{\alpha_i+\gamma_{i-1}-2-\epsilon_{i-1}}2$. Therefore, by (O1), each summand in \eqref{eq:kchaincmuq} is a positive palindromic unimodal polynomial with center of symmetry
\begin{equation}
\frac{\alpha_1-1}2+\sum_{i=2}^{\ell+1}\frac{\alpha_i+\gamma_{i-1}-2-\epsilon_{i-1}}2=\frac{|\alpha|+|\gamma|-2\ell-1-|\epsilon|}2=\frac{2|\gamma|-3\ell-|\epsilon|}2,
\end{equation}
which does not depend on $\alpha$, and therefore we can apply (O2) to see that the sum is a positive palindromic unimodal polynomial with the same center of symmetry. Then we apply (O1) again to see that $c_\mu(q)$ is positive and unimodal.
\end{proof}

\section{Further directions}\label{section:further}

We conclude by proposing some further avenues of study for the Stanley--Stembridge conjecture. First, we can continue to seek sign-reversing involutions on forest triples to prove $e$-positivity and $e$-unimodality for other graphs. The author checked by computer that every natural unit interval graph with at most $9$ vertices has a nice involution.

\begin{problem}
Let $G$ be a natural unit interval graph. Find a nice involution for $G$ to prove that $X_G(\bm x;q)$ is $e$-positive. Furthermore, by enumerating the fixed points, find an explicit $e$-expansion to prove that $X_G(\bm x;q)$ is $e$-unimodal.
\end{problem}


For example, Dahlberg \cite[Theorem 4.5]{trilad} proved that for the class of \emph{triangular ladders}
\begin{equation}
P_{n,2}=([n],\{\{i,j\}: \ 1\leq j-2\leq i\leq j\leq n\}),
\end{equation}
the chromatic symmetric function $X_{P_{n,2}}(\bm x)$ is $e$-positive. We could try to strengthen this result by finding a nice involution to prove a quasisymmetric version. More specifically, because triangular ladders are built by successively joining a new vertex to two vertices, we could take an inductive approach.

\begin{problem}
Let $G=([n],E)$ be a natural unit interval graph with $\{1,2\}\in E$ and let $\varphi$ be a nice involution for $G$. Find a nice involution for the natural unit interval graph
\begin{equation}
([n+1],\{\{1,2\},\{1,3\}\}\cup\{\{i+1,j+1\}: \ \{i,j\}\in E\}).
\end{equation}
Furthermore, by enumerating the fixed points, prove that $X_{P_{n,2}}(\bm x;q)$ is $e$-unimodal. 
\end{problem}

Much of the progress toward the Stanley--Stembridge conjecture considers a particular natural unit interval graph $G$ and shows that every coefficient $c_\mu(q)$ of the $e$-expansion $X_G(\bm x;q)=\sum_\mu c_\mu(q)e_\mu$ is a positive polynomial. We could also consider a dual approach where we fix a particular partition $\mu$ and show positivity for every natural unit interval graph $G$. Upcoming work of Sagan and the author \cite{chrombasis} proves that $c_\mu(q)$ is positive whenever $\mu_1\leq 3$. Hwang \cite[Theorem 5.13]{chromtwoparthook} proved that $c_\mu(q)$ is positive whenever $\mu$ is a \emph{hook}, meaning that $\mu_2\leq 1$. In the $q=1$ case, Abreu and Nigro \cite[Corollary 1.10]{chromhesssplitting} proved that $c_\mu(1)$ is positive whenever $\mu=(n-k)k$ has two parts. We could try to strengthen this result by using forest triples to prove a quasisymmetric version. 

\begin{proposition}
Let $G=([n],E)$ be a natural unit interval graph, let $s(G)=[b_2]_q\cdots[b_n]_q$, using the notation of Example \ref{ex:coefn}, and for $1\leq k\leq n-1$, define 
\begin{equation}s_k(G)=\sum_{\substack{T_1,T_2\text{ dec. trees}, \ V(T_1)\sqcup V(T_2)=[n]\\1\in V(T_1), \  |V(T_2)|=k}}q^{\text{inv}_G(\text{list}(T_1)\cdot\text{list}(T_2))}.
\end{equation} Then for $1\leq k<\frac n2$, the coefficient of $e_{(n-k)k}$ in $X_G(\bm x;q)$ is 
\begin{align}
c_{(n-k)k}(q)&=[n-k]_q[k]_qs_k(G)+[n-k]_q[k]_qs_{n-k}(G)-([n-k]_q+[k]_q)s(G)\\\label{eq:twopartline}&=[n-k]_q([k]_qs_k(G)-s(G))+[k]_q([n-k]_qs_{n-k}(G)-s(G)).
\end{align}
If $k=\frac n2$, then both terms in \eqref{eq:twopartline} are equal and the coefficient is given by just one of them.
\end{proposition}

\begin{proof}
Forest triples $\mathcal F\in\text{FT}_{(n-k)k}(G)$ come in one of two flavours. We could have a single tree triple $\mathcal T=(T,\alpha,r)$, where $T$ is a decreasing spanning tree, contributing a factor of $s(G)$ by Lemma \ref{lem:sumovertreelists}, and $\alpha$ is either $(n-k)k$ or $k(n-k)$. Alternatively, $\mathcal F$ can consist of two tree triples $\mathcal T=(T,\alpha,r)$ and $\mathcal T'=(T',\alpha',r')$ with trees of sizes $(n-k)$ and $k$ and compositions the single parts $(n-k)$ and $k$. Now the result follows from Theorem \ref{thm:foresttriples}.
\end{proof}

Therefore positivity of $c_{(n-k)k}(q)$ would follow from the following. The author checked by computer that this holds for every natural unit interval graph with at most $10$ vertices. 

\begin{problem}
Let $G=([n],E)$ be a natural unit interval graph. Prove that for every $1\leq k\leq n-1$, the polynomial $[k]_qs_k(G)-s(G)$ has positive coefficients.
\end{problem}

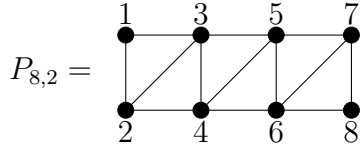
\begin{figure}
\caption{\label{fig:moregraphs} The triangular ladder $P_{8,2}$}
$$
\begin{tikzpicture}
\draw (-5,0.5) node (){$P_{8,2}=$};
\filldraw (-4,1) circle (3pt) node[align=center,above](1){1};
\filldraw (-4,0) circle (3pt) node[align=center,below](2){2};
\filldraw (-3,1) circle (3pt) node[align=center,above](3){3};
\filldraw (-3,0) circle (3pt) node[align=center,below](4){4};
\filldraw (-2,1) circle (3pt) node[align=center,above](5){5};
\filldraw (-2,0) circle (3pt) node[align=center,below](6){6};
\filldraw (-1,1) circle (3pt) node[align=center,above](7){7};
\filldraw (-1,0) circle (3pt) node[align=center,below](8){8};
\draw (-4,1)--(-1,1)--(-1,0)--(-4,0)--(-4,1) (-3,1)--(-3,0) (-2,1)--(-2,0) (-4,0)--(-3,1) (-3,0)--(-2,1) (-2,0)--(-1,1);
\end{tikzpicture}$$

\end{figure}

Another property that is actively studied in relation to unimodality is \emph{log-concavity} \cite{logconuni}. 

\begin{definition}
A polynomial $a(q)=\sum_{k=i}^ja_kq^k\in\mathbb N[q]$ is \emph{log-concave} if for every $k$, we have $a_k^2\geq a_{k-1}a_{k+1}$.  
\end{definition}

Note that if $a(q)$ is a positive log-concave polynomial, then it must be unimodal because we cannot have $a_{k-1}>a_k$ and $a_k<a_{k+1}$. Therefore, log-concavity is a stronger property than unimodality. Thank you to Bruce Sagan for suggesting the following Conjecture. The author checked by computer that this holds for every natural unit interval graph with at most $10$ vertices.

\begin{conjecture} \label{conj:logconcave}
Let $G$ be a natural unit interval graph. Then the chromatic quasisymmetric function $X_G(\bm x;q)$ is \emph{$e$-log-concave}, meaning that for every $\mu$, the coefficient $c_\mu(q)$ of $e_\mu$ is a log-concave polynomial.
\end{conjecture}

Recall that if $G$ is not a natural unit interval graph, then $X_G(\bm x;q)$ is not generally symmetric in the $x_i$ variables, so it may not have an $e$-expansion. However, when $q=1$, the chromatic symmetric function $X_G(\bm x)=X_G(\bm x;1)$ is always symmetric and we can use a variant of forest triples, with a slightly different condition on our trees, to calculate a signed $e$-expansion. We conclude by describing this variant. Let $G=([n],E)$ be an arbitrary graph with a fixed total ordering $\lessdot$ of the edge set $E$.

\begin{definition}
A \emph{no-broken-circuit (NBC) tree} of $G$ is a subtree $T$ of $G$ such that $E(T)$ does not contain a subset of the form $C\setminus\{\max(C)\}$, called a \emph{broken circuit}, where $C$ is a set of edges that forms a cycle and $\max(C)$ is the largest edge of $C$ under $\lessdot$. 
\end{definition}

\begin{remark} If $G$ is a natural unit interval graph, we can define $\lessdot$ \emph{lexicographically}, meaning that for edges $\{i,j\}$ and $\{i',j'\}$ with $i<j$ and $i'<j'$, we define $\{i,j\}\lessdot\{i',j'\}$ if either $i<i'$, or $i=i'$ and $j<j'$. Then NBC trees are precisely decreasing trees. If $i\in V(T)$ has two larger neighbours $j<k$, then $\{j,k\}\in E(G)$ by \eqref{eq:uicondition} and $E(T)$ contains the broken circuit $\{\{i,j\},\{i,k\},\{j,k\}\}\setminus\{\{j,k\}\}$. Conversely, if $E(T)$ contains a broken circuit $C\setminus\{\max(C)\}$, then the smallest vertex in the edges of $C$ has two larger neighbours in $T$. 
\end{remark}

\begin{definition}
A \emph{tree triple} of $G$ is an object $\mathcal T=(T,\alpha,r)$ consisting of the following data.
\begin{itemize}
\item $T$ is an NBC tree of $G$.
\item $\alpha=\alpha_1\cdots\alpha_\ell$ is an integer composition with size $|\alpha|=|V(T)|$.
\item $r$ is a positive integer with $1\leq r\leq\alpha_1$.
\end{itemize}
A \emph{forest triple} of $G$ is a sequence of tree triples $\mathcal F=(\mathcal T_i=(T_i,\alpha^{(i)},r_i))_{i=1}^m$ such that each vertex of $G$ is in exactly one tree $T_i$ and we have
\begin{equation}
\min(V(T_1))<\min(V(T_2))<\cdots<\min(V(T_m)).
\end{equation}
We also define the \emph{type} and \emph{sign} of $\mathcal F$, and the sets $\text{FT}(G)$ and $\text{FT}_\mu(G)$, as in Definition \ref{def:ft}. 
\end{definition}

\begin{theorem}\label{thm:nbctriples}
The chromatic symmetric function $X_G(\bm x)$ of an arbitrary graph $G$ satisfies
\begin{equation}
X_G(\bm x)=\sum_{\mathcal F\in\text{FT}(G)}\text{sign}(\mathcal F)e_{\text{type}(\mathcal F)}=\sum_\mu\left(\sum_{\mathcal F\in\text{FT}_\mu(G)}\text{sign}(\mathcal F)\right)e_\mu.
\end{equation}
In particular, the $e$-expansion does not depend on the choice of total ordering $\lessdot$. 
\end{theorem}

\begin{figure}
\caption{\label{fig:claw}The claw graph $G$ and the chromatic symmetric function $X_G(\bm x)$}
$$\begin{tikzpicture}
\draw (2.5,0) node (0){$G=$};
\draw (3.5,0.866)--(4,0) (3.5,-0.866)--(4,0) (5,0)--(4,0);
\filldraw (3.5,0.866) circle (3pt) node[align=center,above] (1){1};
\filldraw (3.5,-0.866) circle (3pt) node[align=center,below] (2){2};
\filldraw (4,0) circle (3pt) node[align=center,below] (3){3};
\filldraw (5,0) circle (3pt)node[align=center,below] (4){4};
\end{tikzpicture}$$
\begin{equation}
X_G(\bm x)=e_{211}-2e_{22}+5e_{31}+4e_4
\end{equation}
\end{figure}
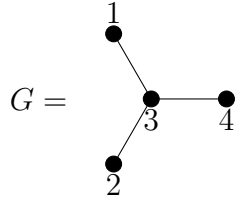

\begin{example}
Let us calculate the coefficient of $e_{22}$ in $X_G(\bm x)$ for the \emph{claw graph} $G$ in Figure \ref{fig:claw}. Because $G$ has no cycles, every subtree is an NBC tree. Note that we cannot break $G$ into two trees of size $2$, because both trees would need to include vertex $3$. Therefore, the only possibility is for $\mathcal F$ to consist of a single tree triple $\mathcal T=(T,\alpha,r)$, where $T$ is all of $G$, $\alpha=22$, and $1\leq r\leq 2$. We have $\text{sign}(\mathcal F)=(-1)^{2-1}=-1$, so these forest triples give us the term $-2e_{22}$. Note that $X_G(\bm x)$ is not $e$-positive.
\end{example}

To prove Theorem \ref{thm:nbctriples}, we use a power sum expansion of $X_G(\bm x)$ and then we perform a change-of-basis from power sum to elementary symmetric functions. The technique of changing bases is also employed in upcoming work of Sagan and the author \cite{chrombasis}. Stanley originally proved the following result using M\"obius inversion, and more recently a sign-reversing involution proof was given by Sagan and Vatter \cite[Section 3]{bijectivenbc}.

\begin{theorem}\label{thm:nbc} \cite[Theorem 2.9]{chromsym}
The chromatic symmetric function $X_G(\bm x)$ is given by
\begin{equation}\label{eq:nbc}
X_G(\bm x)=\sum_{\substack{T_1,\ldots,T_m\text{ NBC trees}\\V(T_1)\sqcup\cdots\sqcup V(T_m)=[n]}}(-1)^{n-m}p_{|V(T_1)|}\cdots p_{|V(T_m)|}.
\end{equation}
\end{theorem}

Mendes and Remmel proved the following change-of-basis formula by showing that both sides satisfy the same recurrence.

\begin{lemma} \label{lem:ptoe} \cite[Theorem 2.22]{countingsym} We have the identity
\begin{equation}\label{eq:ptoe}
p_n=\sum_{\alpha\vDash n}(-1)^{n-\ell(\alpha)}\alpha_1e_{\text{sort}(\alpha)}.
\end{equation}
\end{lemma}

\begin{proof}[Proof of Theorem \ref{thm:nbctriples}. ] By Theorem \ref{thm:nbc} and Lemma \ref{lem:ptoe}, we have 
\begin{align}
X_G(\bm x)&=\sum_{\substack{T_1,\ldots,T_m\text{ NBC trees}\\V(T_1)\sqcup\cdots\sqcup V(T_m)=[n]}}(-1)^{n-m}p_{|V(T_1)|}\cdots p_{|V(T_m)|}\\\nonumber&=\sum_{\substack{T_1,\ldots,T_m\text{ NBC trees}\\V(T_1)\sqcup\cdots\sqcup V(T_m)=[n]}}(-1)^{n-m}\prod_{i=1}^m\sum_{\alpha^{(i)}\vDash |V(T_i)|}(-1)^{|V(T_i)|-\ell(\alpha^{(i)})}\alpha^{(i)}_1e_{\text{sort}(\alpha^{(i)})}\\\nonumber&=\sum_{\substack{T_1,\ldots,T_m\text{ NBC trees}\\V(T_1)\sqcup\cdots\sqcup V(T_m)=[n]\\\alpha^{(1)}\vDash |V(T_1)|,\ldots,\alpha^{(m)}\vDash|V(T_m)|\\1\leq r_1\leq\alpha^{(1)}_1,\ldots,1\leq r_m\leq\alpha^{(m)}_1}}(-1)^{n-m+\sum_{i=1}^m(|V(T_i)|-\ell(\alpha^{(i)}))}e_{\text{sort}(\alpha^{(1)}\cdots\alpha^{(m)})}\\\nonumber&=\sum_{\mathcal F\in\text{FT}(G)}\text{sign}(\mathcal F)e_{\text{type}(\mathcal F)}.
\end{align}\end{proof}

Therefore, even for arbitrary graphs $G$, we can study $e$-positivity of the chromatic symmetric function $X_G(\bm x)$ by fixing a total ordering $\lessdot$ and seeking a sign-reversing involution on forest triples of $G$.

\section{Acknowledgments}
The author would like to thank Per Alexandersson, Alexander Postnikov, Bruce Sagan, Richard Stanley, Aarush Vailaya, and Stephanie van Willigenburg for helpful discussions.

\printbibliography

\end{document}